\documentclass{article}
\usepackage[a4paper,margin=1in]{geometry}  
\usepackage{amsmath, amssymb, amsthm, bm, mathtools} 
\usepackage{graphicx}  
\usepackage{hyperref}  
\usepackage[numbers]{natbib}  
\usepackage{authblk}  
\usepackage{abstract}  
\usepackage{times}  
\usepackage{algorithm}
\usepackage{algpseudocode}
\usepackage{caption}
\usepackage{multirow}
\usepackage{booktabs} 

\newcommand{\R}{\mathbb{R}} 
\newcommand{\N}{\mathbb{N}} 
\newcommand{\C}{\mathbb{C}} 

\newcommand{\E}{\mathbb{E}} 
\renewcommand{\P}{\mathbb{P}} 

\newcommand{\inner}[2]{\left\langle#1, #2\right\rangle} 

\newtheorem{defn}{Definition}
\newtheorem{theo}{Theorem}
\newtheorem{prop}{Proposition}
\newtheorem{lem}{Lemma}
\newtheorem{rem}{Remark}

\newtheorem{assu}{Assumption}

\renewcommand{\H}{\mathcal{H}}
\newcommand{\U}{\mathcal{U}}
\newcommand{\G}{\mathcal{G}}
\newcommand{\B}{\mathcal{B}}
\newcommand{\X}{\mathcal{X}}
\newcommand{\Y}{\mathcal{Y}}
\renewcommand{\S}{\mathcal{S}}
\newcommand{\K}{\mathcal{K}}

\title{Koopman Kalman Filter (KKF): An asymptotically optimal nonlinear filtering algorithm with error bounds and its application to parameter estimation}
\author[1]{Diego Olguín}
\author[1]{Axel Osses}
\author[1]{Héctor Ramírez}
\affil[1]{Departamento de Ingeniería Matemática and Centro de Modelamiento Matemático (CNRS IRL2807), Universidad de Chile, Santiago, Chile}
\date{}  

\begin{document}
\maketitle

\begin{abstract}
    In this article, we propose a new filtering algorithm based in the Koopman operator, showing that a nonlinear filtering problem can be seen as an equivalent problem where the dynamics is infinite dimensional, but linear. Using Extended Dynamic Mode Decomposition (EDMD), we create a finite dimensional approximation of the filtering problem of dimension $N$, in state and error covariance matrix, that accomplishes an error bound of order \(O(N^{-1/2})\) in both where $N$ denotes the number of points used in the Koopman approximation. The algorithm is denominated Koopman Kalman Filter (KKF), and has computational complexity \(O(T\cdot N^3)\) in time, and \(O(T \cdot N^2)\) in space, where \(T\) is the number of iterations of the filtering problem. We test the algorithm in linear and nonlinear dynamics cases, showing and effective error bound with respect to the Kalman filter, that corresponds to the optimal solution in the linear case, and equals the error performance of other methods in the state of the art, but with a much lower execution time. Also, we propose a parameter estimation algorithm based in KKF, comparing it with Markov Chain Monte Carlo techniques, showing similar performance with lower execution time. 
\end{abstract}


\section{Introduction}
The filtering problem concerns the estimation of the trajectory of a dynamical system, potentially stochastic, that is typically not fully observable and subject to noisy measurements. This problem arises in diverse areas of engineering and science, including signal processing, astronomy, fluid mechanics, meteorology, medical imaging, and the study of epidemiology and infectious disease dynamics, among others \cite{shumway_time_2025}.

The first mathematical advances on the filtering problem were made independently by Wiener \cite{Weiner1950Extrapolations} and Kolmogorov \cite{Kolmogorov1940StationarySpace}, who proposed general statistical approximations. The first algorithmic solution, the Kalman filter, was introduced by Kalman \cite{Kalman1960AProblems}, providing an exact solution in discrete time for the linear case.

In parallel with Kalman’s work on the linear setting, Stratonovich investigated the general nonlinear case \cite{Stratonovich1965ApplicationSignals}. He showed that the problem reduces to solving a stochastic differential equation whose solution coincides with Kalman’s result in discrete time and with the Kalman–Bucy filter in continuous time \cite{Kalman1961NewTheory}. However, in the most general case, Maurel and Michel \cite{Maurel1984DesFinie} demonstrated that the solutions to the Stratonovich equation admit an infinite-dimensional representation, rendering them intractable for computation and motivating the search for practical approximations.

To address the case of nonlinear dynamics, several methodologies have been developed to obtain suboptimal approximations to the filtering problem. Among the most prominent are the Extended Kalman Filter (EKF) \cite{McElhoe1966AnVenus, Smith1962ApplicationVehicle} and the Unscented Kalman Filter (UKF) \cite{Julier2004UnscentedEstimation}. Although these methods extend the classical Kalman filter, they may, in certain cases, produce unsatisfactory results and lack the asymptotic properties necessary to ensure convergence to the optimal solution.

Another widely adopted strategy is the use of Particle Filters \cite{Hammersley1954PoorCarlo, Liu1998SequentialSystems}, which rely on Monte Carlo methods to approximate the solution while preserving the sequential structure of the problem. Owing to their asymptotic nature, and under appropriate conditions on the system and particle resampling, particle filters exhibit error bounds of order $O(N_p^{-1/2})$, where $N_p$ denotes the number of particles, for the solution of the filtering problem \cite{Crisan2002APractitioners}.

More recently, researchers have investigated the formulation of the filtering problem within Reproducing Kernel Hilbert Spaces (RKHS). These spaces, well known for their applications in statistics, machine learning, and numerical analysis \cite{Wendland2004ScatteredApproximation, Christmann2008SupportMachines}, possess powerful interpolation and approximation properties. Among the most significant contributions in this direction are those of Fukumizu and Song \cite{Fukumizu2004DimensionalitySpaces, Song2009HilbertSystems}, who introduced the Kernel Bayes rule, a generalization of the classical Bayes rule adapted to Hilbert space settings.

On the other hand, data-driven methods applied to dynamical systems have found a powerful framework in the Koopman operator. Originally introduced by Koopman and von Neumann \cite{Koopman1931HamiltonianSpace, Koopman1932DynamicalSpectra}, this operator enables the representation of finite-dimensional nonlinear dynamical systems as infinite-dimensional linear ones. Building on this foundation, the contributions of Mezić \cite{Mezic2013AnalysisOperator}, Schmid \cite{Schmid2008DynamicData}, Surana \cite{Surana2016KoopmanSystems}, Brunton \cite{Brunton2016KoopmanControl}, and others have significantly expanded the use of the Koopman operator, particularly in fields such as fluid mechanics, where rich datasets facilitate the identification of complex dynamics.

In this article, we propose an algorithm that integrates the Koopman operator with the Kalman filter: the Koopman Kalman Filter (KKF). The method achieves an error bound of $O(N^{-1/2})$, where $N$ denotes the number of points used in the Koopman approximation. Although the algorithm is not entirely new and similar formulations have been empirically proposed by other authors \cite{HuangData-DrivenFlight, Netto2018RobustEstimation, Syed2021Koopman-basedXFEL, Wang2022KoopmanSystem, Wang2023Innovation-saturatedOutliers, Yang2025Data-DrivenPredictor}, the novelty of our work lies in deriving the filter rigorously from the classical solution of the filtering problem together with its error bound, and in proving that is competitive with other nonlinear filters in relevant applications and also can be useful in parameter estimation of dynamical systems.

The remainder of this paper is organized as follows. In Section 2, we review the necessary preliminaries for the filtering problem in the linear case, the Kalman filter, and the relevant elements of Reproducing Kernel Hilbert Spaces, following our previous work \cite{olguin2025improvederrorboundskoopman}. Section 3 presents our main results, structured into four parts: (i) the decomposition of the Kalman filter in the linear setting, (ii) the derivation of the filter equations in RKHS, (iii) the establishment of the corresponding error bound, and (iv) an application of the filter to parameter estimation. Section 4 provides numerical results for both filtering and parameter estimation, demonstrating improved precision in filtering compared to existing algorithms, as well as superior computational efficiency in parameter estimation a typically demanding task in terms of computational resources. 

\section{Preliminaries}
\subsection{Filtering problem}

Consider an autonomous and stochastic dynamical and observerd system, which can be 
represented by the following equations:
\[
\begin{aligned}
    \mathbf{x}_{k+1} &= \mathbf{f}\left(\mathbf{x}_k, \mathbf{w}_k\right) \\
    \mathbf{y}_{k} &= \mathbf{g}\left(\mathbf{x}_k, \mathbf{v}_k\right).
\end{aligned}
\]
Here, $\mathbf{f}: \R^n \times \Omega_\mathbf{w} \to \R^n$, $\mathbf{g}: \R^n \times \Omega_\mathbf{v} \to \R^p$ where:
\begin{itemize}
    \item $\mathbf{x}_k$ represents the state of the system at discrete time step 
$k$.
    \item  $\mathbf{w}_k$ is a random variable with support equal to 
$\Omega_\mathbf{w}$, introducing stochasticity into the dynamics. 
    \item  $\mathbf{v}_k$ is a random variable with support equal to 
$\Omega_\mathbf{v}$, introducing stochasticity into the observations.
\end{itemize}

The filtering problem up to a finite horizon $T>0$ can be formulated as an optimization problem of the mean square error of the estimation, conditional on the observations and inputs, given a prior distribution for the initial state $X_0\in\R^n$. So, this is:
\begin{equation*}
	(P) \quad
	\begin{cases}
		\begin{aligned}
			\text{min} \quad & \quad \sum_{k=0}^T \mathbb{E} \left [ (\hat{\mathbf{x}}_{k|k} - \mathbf{x}_k)^\top(\hat{\mathbf{x}}_{k|k} - \mathbf{x}_k) | \mathbf{y}_{1:k}, \, \mathbf{u}_{0:k} \right] \\
			\text{s.t.} \quad & \quad \mathbf{x}_{k+1} = \mathbf{f}( \mathbf{x}_k, \mathbf{w}_k) \\
			\text{} \quad & \quad \mathbf{y}_k = \mathbf{g}(\mathbf{x}_k, \mathbf{v}_k) \\
			\text{} \quad & \quad \mathbf{x}_0 \sim X_0.
		\end{aligned}
	\end{cases}
\end{equation*}
It is then that the problem $(P)$ has as its solution the minimum mean square error estimator. It is known that this problem has a global optimum given by the conditional expectation.
\begin{prop}[\cite{Kalman1960AProblems}]
	The optimum of $(P)$ is given by
	\begin{equation*}
		\hat{\mathbf{x}}_{k|k} = \mathbb{E} \left [ \mathbf{x}_k | \mathbf{y}_{1:k} \right], \, \forall k \in \{ 0, \dots, T \}.
	\end{equation*}
\end{prop}

\subsection{Linear case and Kalman filter}

To address the general case of dynamical systems, it is instructive to first study the linear case with centered noises and finite second moments, namely
\begin{equation}
\begin{cases}
\mathbf{x}_{k+1} &= \mathbf{A}_k \mathbf{x}_k + \mathbf{w}_k, \\
\mathbf{y}_k &= \mathbf{C}_k \mathbf{x}_k + \mathbf{v}_k,
\end{cases}
\label{eq:lin_disc}
\end{equation}
where $\mathbf{A}_k \in \mathbb{R}^{n \times n}$, $\mathbf{C}_k \in \mathbb{R}^{p \times n}$, $\mathbb{E}[\mathbf{w}_k] = 0$, $\mathrm{Cov}(\mathbf{w}_k) = \mathbf{Q}_k$, $\mathbb{E}[\mathbf{v}_k] = 0$, and $\mathrm{Cov}(\mathbf{v}_k) = \mathbf{R}_k$, with $\mathbf{Q}_k \in \mathbb{R}^{n \times n}$ and $\mathbf{R}_k \in \mathbb{R}^{p \times p}$ covariance matrices.

In this setting, the Kalman filter, introduced by Kalman in \cite{Kalman1960AProblems}, provides the first algorithmic solution to the filtering problem. Its pseudo-code is summarized in Algorithm~\ref{alg:KF} and is important for the one proposed in this article. The method consists of two steps: a \textbf{prediction} step, in which the state is propagated through the system dynamics to generate an a priori estimate, and an \textbf{update} step, in which the observations are used to correct this estimate. Kalman, also in \cite{Kalman1960AProblems} further proved that the algorithm yields the optimal solution to the problem, namely the conditional expectation.

\begin{prop}[\cite{Kalman1960AProblems}]
The Kalman filter produces a sequence $(\hat{\mathbf{x}}_{k|k})$ such that
\begin{equation*}
\hat{\mathbf{x}}_{k|k} = \mathbb{E} \left[ \mathbf{x}_k \middle| \mathbf{y}_{1:k} \right], \quad \forall k \in {0, \dots, T}.
\end{equation*}
\end{prop}
\begin{algorithm}
	\caption{Kalman Filter}\label{alg:KF}
	\begin{algorithmic}[1]
		\State \textbf{Input:} Discrete dynamics as in (\ref{eq:lin_disc}), $\mathbf{x}_0$ prior on the initial condition, $\mathbf{y}_{1:T}$ observations, $\mathbf{u}_{0:T}$ inputs.
		\State \textbf{Output:} $(\hat{\mathbf{x}}_{k|k})_{k=0}^{T}$ trajectory estimator and $(\hat{\mathbf{P}}_{k|k})_{k=0}^{T}$ covariance matrices.
		\State $\hat{\mathbf{x}}_{0|0} \gets \mathbb{E} [\mathbf{x}_0]$
		\State $\mathbf{P}_{0|0} \gets \mathbb{E} [(\mathbf{x}_0 - \hat{\mathbf{x}}_{0})(\mathbf{x}_0 - \hat{\mathbf{x}}_{0})^\top]$
		\For{$k = 0, \dots, T-1$}
		\State $\hat{\mathbf{x}}_{k+1|k} \gets \mathbf{A}_k \mathbf{x}_{k|k}$
		\Comment{A priori estimation}
		\State $\mathbf{P}_{k+1|k} \gets \mathbf{A}_k \mathbf{P}_{k|k} \mathbf{A}_k^\top + \mathbf{Q}_k$
		\Comment{A priori error covariance}
		\State $\hat{\mathbf{y}}_{k+1|k} \gets \mathbf{C}_{k+1} \hat{\mathbf{x}}_{k+1|k}$
		\Comment{A priori observation estimation}
		\State $\mathbf{e}_{\mathbf{y}_{k+1|k}} \gets \mathbf{y}_{k+1} - \hat{\mathbf{y}}_{k+1|k}$
		\Comment{A priori error (innovation)}
		\State $\mathbf{K}_{k+1} \gets \mathbf{P}_{k+1|k} \mathbf{C}^T_{k+1} (\mathbf{C}_{k+1} \mathbf{P}_{k|k} \mathbf{C}^T_{k+1}+ \mathbf{R}_{k+1})^{-1}$
		\Comment{Kalman gain}
		\State $\hat{\mathbf{x}}_{k+1|k+1} \gets \hat{\mathbf{x}}_{k+1|k} + \mathbf{K}_{k+1} \mathbf{e}_{\mathbf{y}_{k+1|k}}$
		\Comment{A posteriori estimation}
		\State $\mathbf{P}_{k+1|k+1} \gets (\mathbf{I} - \mathbf{K}_{k+1} \mathbf{C}_{k+1}) \mathbf{P}_{k+1|k}$
		\Comment{A posteriori error covariance}
		\EndFor
	\end{algorithmic}
\end{algorithm}

\subsection{Setting for Koopman operator approximation}

In this subsection we present the framework and notation for the Koopman operator approximation, that is the most important tool in this work. Similar settings can be found in literature, like \cite{Philipp2023ErrorFramework} and \cite{Philipp2024ErrorOperator}, but the most important for this case is \cite{olguin2025improvederrorboundskoopman}, where an error bound for the approximation is stablished and will be used later. 

First, consider the Borel $\sigma$-algebras over $\R^n$ and $\R^p$, denoted by $\B_{\R^n}$ and $\B_{\R^p}$, respectively, together with measures that model the transition from a point to a set, referred to as \textit{transition measures}.  

\begin{defn}[Transition measures]  
The \textit{transition measures} are defined as the probability measures  
\(\rho_\mathbf{f}: \R^n \times \B_{\R^n} \to [0, 1], \,  
\rho_\mathbf{g}: \R^n \times \B_{\R^p} \to [0, 1],\) acting as  
\[
\rho_\mathbf{f}(\mathbf{x}, \mathcal{A}_1) = \P \bigl(\mathbf{f}(\mathbf{x}, \cdot) \in \mathcal{A}_1 \bigr),  
\qquad  
\rho_\mathbf{g}(\mathbf{x}, \mathcal{A}_2) = \P \bigl(\mathbf{g}(\mathbf{x}, \cdot) \in \mathcal{A}_2 \bigr).
\]  
We adopt the shorthand notation  
\[
\rho_\mathbf{f}(\mathbf{x}, dx') = d\rho_\mathbf{f}(\mathbf{x}, \cdot)(x'),  
\qquad  
\rho_\mathbf{g}(\mathbf{x}, dy') = d\rho_\mathbf{g}(\mathbf{x}, \cdot)(y').  
\]  
\end{defn}  

An important object in this work, is the concept of a pair of invariant spaces for the system, that are sets where the dynamical system remains in the next unit of time with probability 1.

\begin{defn}[Invariant spaces]  
A pair $(\X, \Y)$ with $\X \subseteq \R^n$ and $\Y \subseteq \R^p$ is called a pair of \textit{invariant spaces} for the observed dynamical system if  
\[
\rho_\mathbf{f}(\mathbf{x}, \X) = 1, \, \rho_\mathbf{g}(\mathbf{x}, \Y) = 1, \qquad \forall \mathbf{x} \in \X.  
\]  
\end{defn} 

Before defining the Koopman operator, we need an appropriate functional space in which to work. In our case, we shall employ Reproducing Kernel Hilbert Spaces (RKHS), that have been relevant in the literature of Koopman operator in recent times and where the error bounds appear with this setting. 

\begin{defn}[Reproducing Kernel Hilbert Space (RKHS) \cite{Mercer1909XVI.Equations}]  
A Hilbert space \( \H \subset \mathbb{C}^{\X} \) is called an RKHS if there exists a function \( k: \X \times \X \to \mathbb{C} \), called the \textit{reproducing kernel}, such that:  
\begin{enumerate}
    \item \( k_p \equiv k(\cdot, p) \in \H, \ \forall p \in \X. \)
    \item \( f(p) = \langle f, k_p \rangle_{\H}, \ \forall p \in \X, \ \forall f \in \H. \)
\end{enumerate}
\end{defn}  

The second property, known as the \textit{reproducing property}, is fundamental in RKHS theory and gives rise to many important results, as will be discussed below. In the sequel, we denote by $\H$ an RKHS associated with a kernel $k$.  

\begin{theo}  
A function \( k: \X \times \X \to \mathbb{C} \) is a \textit{reproducing kernel} of a Hilbert space $\H$ if and only if there exists a function \( \Phi: \X \to \ell^2 (\X) \) such that, for all \( (x, y) \in \X \times \X \),  
\[
k(x, y) = \langle \Phi(x), \Phi(y) \rangle_{\ell^2(\X)}.
\]  
Moreover, the space \( \ell^2(\X) \) is isometric to $\H$ via an isometry \( T: \H \to \ell^2(\X) \), and the function \( \Phi \) is given by \( \Phi(x) = T(k(\cdot, x)) \). This function \( \Phi \) is called the \textit{feature map} associated with \( k \).  
\end{theo} 

We consider \( (\X, \Y) \) a pair of invariant spaces of the observed dynamical system.  
Let \( \B_\X \) denote the trace \( \sigma \)-algebra of \( \B_{\R^n} \) on \( \X \), and \( \B_\Y \) the trace of \( \B_{\R^p} \) on \( \Y \).  
We define \( \mu_\X: \B_\X \to [0,1] \) as a probability measure representing a random variable \( X \), from which the data are sampled. Similarly, let \( \mu_\Y : \B_\Y \to [0,1] \) be a probability measure on \( \Y \).  

We denote by \( X^+ \) the random variable that, given \( X = \mathbf{x} \), is distributed according to the measure \( \rho_\mathbf{f}(\mathbf{x}, \cdot) \).  
Likewise, we denote by \( Y \) the random variable that, given \( X = \mathbf{x} \), is distributed according to \( \rho_\mathbf{g}(\mathbf{x}, \cdot) \). Finally, let \( \Phi_\X \) denote the canonical feature map associated with a kernel \( k_\X : \X \times \X \to \R \), i.e., \(\Phi_\X(x) = k_\X(x, \cdot)\).

Now arise the first operators, the covariance operators, which are the best way to define the Koopman operator in the context of RKHS, because they allow us to understand the power of the operator to linearize the problem, which is the main advantage of this tool. First, we introduce the Kronecker product as a rank-one operator acting as a multiplicative operator.

\begin{defn}[Kronecker Product]
The Kronecker product in an RKHS is defined as the following rank-one operator for fixed \(x, x' \in \X\):
\[
    \Phi_\X (x) \otimes \Phi_\X(x') : \H_\X \to \H_\X,
\]
\[
    [\Phi_\X (x) \otimes \Phi_\X(x')] \psi = \langle \psi, \Phi_\X (x') \rangle_{\H_\X} \, \Phi_\X (x) 
    = \psi (x') \Phi_\X (x).
\]
More generally, it can be defined between two RKHS. For \(x \in \X\) and \(y \in \Y\),
\[
    \Phi_\X (x) \otimes \Phi_\Y(y) : \H_\Y \to \H_\X,
\]
\[
    [\Phi_\X (x) \otimes \Phi_\Y(y)] \psi 
    = \langle \psi, \Phi_\Y (y) \rangle_{\H_\Y} \, \Phi_\X (x) 
    = \psi (y) \Phi_\X (x).
\]
\end{defn}

From this, the covariance operators follow, and we present them as in \cite{Song2009HilbertSystems}, where we define the covariance operator between a distribution \(X\) and itself, called the auto-covariance, and between two different distributions \(X\) and \(X^+\), and \(X\) and \(Y\), which are called cross-covariances.

\begin{defn}[Covariance operators]
The autocovariance operator is the map \( C_X: \H_\X \to \H_\X \) defined by
\[
    C_{X} = \E[\Phi_\X (X) \otimes \Phi_\X (X)] 
    = \int_{\X} \Phi_\X (x) \otimes \Phi_\X (x) \, d \mu_{\X} (x).
\]
And the cross-covariance operator associated with the random variables \(X\) and \(X^+\) is defined as 
\(C_{XX^+}: \H_\X \to \H_\X\),
\[
    C_{XX^+} = \E \big[ \Phi_\X (X) \otimes \Phi_\X (X^+) \big] 
    = \int_{\X} \int_{\X} \Phi_\X(x) \otimes \Phi_\X(x') \, \rho_\mathbf{f} (x, dx') \, d \mu_{\X} (x).
\]
Similarly, the operator associated with \(X\) and \(Y\) is defined as
\[
    C_{XY} = \E \big[ \Phi_\X (X) \otimes \Phi_\Y (Y) \big] 
    = \int_{\X} \int_{\Y} \Phi_\X (x) \otimes \Phi_\Y (y) \, \rho_\mathbf{g} (x, dy) \, d \mu_{\X} (x).
\]
\end{defn}

A fundamental tool, closely related to the Koopman operator, is the \textit{conditional embedding operator}. 
It represents transitions in the state space by embedding conditional expectations into the RKHS, but in a linear manner with respect to the \textit{feature map}. 
This is the key idea underlying the Koopman operator and will be crucial for the algorithm presented next.

\begin{defn}[Conditional Embedding Operator]
The conditional embedding operators between \(X\) and \(X^+\), and between \(X\) and \(Y\), are the maps \(C_{X^+|X} : \H_\X \to \H_\X, \, C_{Y|X} : \H_\X \to \H_\Y,\)defined by the following conditions:
\begin{enumerate}
    \item For each \(x \in \X\),
    \[
        \mu_{X^+|x} = \E_{X^+|X}\!\left[\Phi_\X(X^+)\,\middle|\,X=x\right] 
        = C_{X^+|X} \Phi_\X (x),
    \]
    \[
        \mu_{Y|x} = \E_{Y|X}\!\left[\Phi_\Y(Y)\,\middle|\,X=x\right] 
        = C_{Y|X} \Phi_\X (x).
    \]
    \item For all \(h_\X \in \H_\X\) and \(h_\Y \in \H_\Y\),
    \[
        \E_{X^+|X}[h_\X(X^+)\,|\,X=x] 
        = \langle h_\X, \mu_{X^+|x} \rangle_{\H_\X}, 
        \qquad
        \E_{Y|X}[h_\Y(Y)\,|\,X=x] 
        = \langle h_\Y, \mu_{Y|x} \rangle_{\H_\Y}.
    \]
\end{enumerate}
\end{defn}

A condition for the existence of these operators was established by Fukumizu et al. and Song et al. \cite{Fukumizu2004DimensionalitySpaces, Song2009HilbertSystems}, 
and it is satisfied by a broad family of RKHSs, for instance, those associated with universal kernels.

\begin{theo}
Assume that 
\(\E[h_\X(X^+)\,|\,X=\cdot] \in \H_\X\) for all \(h_\X \in \H_\X\) 
and 
\(\E[h_\Y(Y)\,|\,X=\cdot] \in \H_\Y\) for all \(h_\Y \in \H_\Y\).  
Then the conditional embedding operators satisfy
\[
    C_{X^+|X} = C_{X^+X} \, C_X^{-1}, 
    \qquad 
    C_{Y|X} = C_{YX} \, C_X^{-1}.
\]
\end{theo}

With all the previous tools, we can now define the Koopman operator for autonomous and stochastic systems, in line with other works (see, e.g., \cite{Philipp2023ErrorFramework}).

\begin{defn}[Koopman Operators] 
The Koopman operator associated with the dynamics is the linear operator 
\(\U : \C^\X \to \C^\X \) defined by
\[
    (\U \psi)(\mathbf{x}) 
    = \E\!\left[\psi(\mathbf{f}(\mathbf{x},\cdot))\right] 
    = \int_\X \psi(x') \, \rho_\mathbf{f}(\mathbf{x}, dx').
\]
Similarly, the Koopman operator associated with the observation is 
\(\G : \C^\Y \to \C^\X\), defined by
\[
    (\G \psi)(\mathbf{x}) 
    = \E\!\left[\psi(\mathbf{g}(\mathbf{x},\cdot))\right] 
    = \int_\Y \psi(y) \, \rho_\mathbf{g}(\mathbf{x}, dy).
\]
\end{defn}

At this point, the relationship between the Koopman operator and the covariance operators is not immediately apparent. 
However, the following proposition bridges these concepts, showing in particular that the Koopman operator is the adjoint of the previously introduced Conditional Embedding Operator.

\begin{prop}\label{prop:covariance_koopman}
    Suppose that \(\U \H_\X \subseteq \H_\X\) and \(\G \H_\Y \subseteq \H_\X\).  
    Then the conditional embedding operators satisfy
    \[
        C_{X^+|X} = \U^*, 
        \qquad 
        C_{Y|X} = \G^*,
    \]
    which implies
    \[
        \U = C_X^{-1} C_{XX^+}, 
        \qquad 
        \G = C_X^{-1} C_{XY}.
    \]
\end{prop}

\begin{proof}[Proof of Proposition \ref{prop:covariance_koopman}]
    This result is similar to Proposition 3 in \cite{olguin2025improvederrorboundskoopman}, 
where it was already proven that \(C_{X^+|X} = \U^*\); the case \(C_{Y|X} = \G^*\) follows analogously.
\end{proof}

\section{Main results}
We divide the theoretical results into four parts. First, we propose a decomposition of the error of the Kalman filter applied to two linear systems, expressing this error in terms of the data generated by the systems. To the best of our knowledge, this analysis has not previously been documented in the literature and therefore constitutes an original contribution of this work.  
Second, we rigorously derive the proposed filter, referred to as the \textit{Koopman Kalman Filter} (KKF), following a methodology analogous to that employed in the Kalman filter for linear systems with recursive least squares \cite{Kalman1960AProblems, Triantafyllopoulos2021BayesianBeyond}, and in particular, the formulation presented in \cite{Gebhard2019} for the construction of the so called \textit{kernel Kalman rule}.  
Third, we establish an error bound for the approximated solution of the filter constructed via Extended Dynamic Mode Decomposition (EDMD), measured with respect to the filter in infinite dimensions. Finally, we present a methodology for parameter estimation based in the new filter.

\subsection{Kalman error decomposition}

The objective of this subsection is to analyze the error between two Kalman rules, thereby quantifying the discrepancy between an approximate and an exact rule, both of which will be formally defined later in this section. To this end, we assume a finite time horizon $T$, which avoids concerns regarding the behavior of constants in infinite time horizon.  

We consider two dynamical systems observed in a Hilbert space, with state spaces $E_x$ and observation spaces $E_y$, described by
\begin{equation*}
	\begin{aligned}
		\mu_{i,k}  &= A_{i,k} \mu_{i,k-1} + \nu_{i,k}, \\
		y_{i,k} &= C_{i,k} \mu_{i,k} + \xi_{i,k},
	\end{aligned}
\end{equation*}
where $A_{i,k} : E_x \to E_x$ and $C_{i,k}: E_x \to E_y$ are bounded linear operators, while $\nu_{i,k} \in E_x$ and $\xi_{i,k} \in E_y$ are random variables with finite second moments and bounded covariance operators $\mathcal{Q}_{i,k}$ and $\mathcal{R}_{i,k}$, respectively, for $i \in \{1,2\}$ and $k \geq 1$.  

Each system is associated with a Kalman rule, defined by
\begin{equation*}
	\begin{aligned}
        \hat{\mu}_{i,k+1}^- &= A_{i,k} \hat{\mu}_{i,k}, 
        & \mathcal{P}_{i,k}^- &= A_{i,k} \mathcal{P}_{i,k-1}^+ A_{i,k}^* + \mathcal{Q}_{i,k}, \\
		\S_{i,k} &= C_{i,k} \mathcal{P}_{i,k}^- C_{i,k}^* + \mathcal{R}_{i,k}, 
		& \K_{i,k} &= \mathcal{P}_{i,k}^- C_{i,k}^* \S_{i,k}^{-1}, \\
		\mathcal{P}_{i,k}^+ &= (I - \K_{i,k} C_{i,k}) \mathcal{P}_{i,k}^-, 
		& \hat{\mu}_{i,k} &= A_{i,k} \hat{\mu}_{i,k-1} + \K_{i,k} (y_{i,k} - C_{i,k} \hat{\mu}_{i,k}^-),
	\end{aligned}
\end{equation*}
for $i \in \{1,2\}$ and $k \geq 1$. Here, $\mathcal{P}_{i,k}^-$ and $\mathcal{P}_{i,k}^+$ denote the a priori and a posteriori error covariance operators, respectively, and $\K_{i,k}$ is the Kalman gain operator, all defined on the corresponding spaces. The initialization is given by
\[
    \hat{\mu}_{i,0} = \mathbb{E}[\mu_{i,0}], 
    \qquad 
    \mathcal{P}_{i,0} = \text{Cov}(\mu_{i,0}).
\]

With these definitions in place, we now present a key result that shows how the norm discrepancy between two Kalman rules can be decomposed into contributions from the discrepancies of their associated elements, together with the influence of the previous iteration.  

\begin{theo}[Kalman Error Decomposition]
	Let $k \geq 1$. If the operators $\S_{i,k}$ are invertible, then there exist constants $c_{k,j}^i$ with $j \in \{1, \dots, 7\}$ and $i \in \{1,2\}$ such that the following inequalities hold:
	\begin{equation*}
		\begin{aligned}
			\| \hat{\mu}_{1,k} - \hat{\mu}_{2,k} \| \leq & \; c_{1,k}^1 \| A_{1,k} - A_{2,k} \| + c_{2,k}^1 \| C_{1,k} - C_{2,k} \| + c_{3,k}^1 \| \mathcal{Q}_{1,k} - \mathcal{Q}_{2,k} \| + c_{4,k}^1 \| \mathcal{R}_{1,k} - \mathcal{R}_{2,k} \| \\
            &+ c_{5,k}^1 \| y_{1,k} - y_{2,k} \| + c_{6,k}^1 \| \hat{\mu}_{1,k-1} - \hat{\mu}_{2,k-1} \| + c_{7,k}^1 \| \mathcal{P}_{1,k-1}^+ - \mathcal{P}_{2,k-1}^+ \|,
		\end{aligned}
	\end{equation*}
	and
	\begin{equation*}
		\begin{aligned}
			\| \mathcal{P}_{1,k}^+ - \mathcal{P}_{2,k}^+ \| \leq & \; c_{1,k}^2 \| A_{1,k} - A_{2,k} \| + c_{2,k}^2 \| C_{1,k} - C_{2,k} \| + c_{3,k}^2 \| \mathcal{Q}_{1,k} - \mathcal{Q}_{2,k} \| + c_{4,k}^2 \| \mathcal{R}_{1,k} - \mathcal{R}_{2,k} \| \\
            &+ c_{5,k}^2 \| y_{1,k} - y_{2,k} \| + c_{6,k}^2 \| \hat{\mu}_{1,k-1} - \hat{\mu}_{2,k-1} \| + c_{7,k}^2 \| \mathcal{P}_{1,k-1}^+ - \mathcal{P}_{2,k-1}^+ \|.
		\end{aligned}
	\end{equation*}

	Here, the constants $c_{k,j}^i$ are positive and depend on $k$ only through finite combinations of sums and products of the norms 
	\( A_{i,k} \|, \, \| C_{i,k} \|, \, \| \mathcal{Q}_{i,k} \|, \, \| \mathcal{R}_{i,k} \|, \, \| \S_{i,k}^{-1} \|, \, \| y_{i,k} \|, \, \| \hat{\mu}_{i,k-1} \|, \, \text{and} \, \| \mathcal{P}_{i,k-1}^+ \|.
	\)
	\label{theo:kalman_error}
\end{theo}

\begin{proof}[Proof of Theorem \ref{theo:kalman_error}]
We begin by observing that
\begin{equation*}
	\begin{aligned}
	   \| \hat \mu_{1,k} - \hat \mu_{2,k} \|  
		& \leq \| A_{1,k} \hat \mu_{1,k-1}  - A_{2,k} \hat \mu_{2,k-1} \| \\
        & \quad + \| \K_{1,k} (y_{1,k} - C_{1,k} A_{1,k-1} \hat\mu_{1,k-1}) -  \K_{2,k} (y_{2,k} - C_{2,k} A_{2,k-1} \hat\mu_{2,k-1}) \|.
	\end{aligned}
\end{equation*}

The first term, referred to as the \textit{prediction error}, satisfies
\begin{equation*}
	\| A_{1,k} \hat \mu_{1,k-1} - A_{2,k} \hat \mu_{2,k-1} \|  
	\leq \| A_{1,k} \| \, \| \hat \mu_{1,k-1} - \hat \mu_{2,k-1} \| 
	+  \| \hat \mu_{2,k-1} \| \, \| A_{1,k} - A_{2,k} \|.
\end{equation*}

It remains to bound the second term, referred to as the \textit{update error}. We have
\begingroup
\allowdisplaybreaks
\begin{align*}
    & \| \K_{1,k} (y_{1,k} - C_{1,k} A_{1,k-1}\hat\mu_{1,k-1}) -  \K_{2,k} (y_{2,k} - C_{2,k} A_{2,k-1} \hat\mu_{2,k-1})  \| \\
        & \leq \| \K_{1,k} \| \, \| C_{1,k} \| \, \| A_{1,k} \| \, \|  \hat\mu_{1,k-1} - \hat\mu_{2,k-1}  \| \\
        & \quad + \| \K_{1,k} \| \, \|  y_{1,k} - y_{2,k}  \| 
        + \big( \| y_{2,k} \| + \| \hat\mu_{2,k-1} \| \, \| C_{2,k} \| \big) \| \K_{1,k}  -  \K_{2,k}  \| \\
        & \quad + \| \hat\mu_{2,k-1} \| \, \| \K_{1,k} \| \, \big( \| A_{1,k-1} \| + \| A_{2,k-1} \| \big) \, \| C_{1,k} - C_{2,k} \|.
\end{align*}
\endgroup

Moreover, we have the bound
\[
\| \K_{i,k} \| \leq \| \mathcal{P}_{i,k}^- \| \, \| C_{i,k} \| \, \| \S_{i,k}^{-1} \|
\leq \big( \| A_{i,k} \|^2 \, \| \mathcal{P}_{i,k-1}^+ \| + \| \mathcal{Q}_{i,k} \| \big)\, \| C_{i,k} \| \, \| \S_{i,k}^{-1} \|.
\]

In light of the above, it remains to analyze the norm difference between the gain operators:
\begin{equation*}
	\begin{aligned}
		 \| \K_{1,k}  - \K_{2,k} \| 
	 & \leq \| C_{1,k}\S_{1,k}^{-1} \| \, \| \mathcal{P}_{1,k}^- - \mathcal{P}_{2,k}^- \| \\
     & \quad + \| \mathcal{P}_{2,k}^- \| \Big( \|  C_{1,k} \| \, \| \S_{1,k}^{-1} -  \S_{2,k}^{-1} \| + \| \S_{2,k}^{-1} \| \, \| C_{1,k} -  C_{2,k} \| \Big).
	\end{aligned}
\end{equation*}

Using the relation \(\mathcal{P}_{i,k}^- = A_{i,k} \mathcal{P}_{i,k-1}^+ A_{i,k}^* + \mathcal{Q}_{i,k}\) we have \(\| \mathcal{P}_{i,k}^- \| \leq \| A_{i,k} \|^2 \, \| \mathcal{P}_{i,k-1}^+ \| + \| \mathcal{Q}_{i,k} \|,\) and we obtain
\begin{align*}
	 \| \K_{1,k} - \K_{2,k} \| 
	 & \leq \| C_{1,k} \| \, \| \S_{1,k}^{-1} \| \Big( (\| A_{1,k} \| \, \| \mathcal{P}_{1,k-1}^+ \| + \| A_{2,k} \| \, \| \mathcal{P}_{2,k-1}^+ \| ) \, \| A_{1,k} - A_{2,k} \| + \| \mathcal{Q}_{1,k} - \mathcal{Q}_{2,k} \| \Big) \\
	 & \quad + \big( \| A_{2,k} \|^2 \, \| \mathcal{P}_{2,k-1}^+ \| + \| \mathcal{Q}_{2,k} \| \big) \Big( \| C_{1,k} \| \, \| \S_{1,k}^{-1} - \S_{2,k}^{-1} \| + \| \S_{2,k}^{-1} \| \, \| C_{1,k} - C_{2,k} \| \Big).
\end{align*}
Finally, to analyze $\| \S_{1,k}^{-1} -  \S_{2,k}^{-1} \|$, we use the identity
\(\S_{1,k}^{-1} -  \S_{2,k}^{-1} = \S_{2,k}^{-1} (\S_{2,k} - \S_{1,k}) \S_{1,k}^{-1},\) which yields
\begin{align*}
	\| \S_{1,k}^{-1} -  \S_{2,k}^{-1} \| 
	& \leq \| \S_{2,k}^{-1} (\S_{2,k} - \S_{1,k}) \S_{1,k}^{-1} \| \\
	& \leq \| \S_{1,k}^{-1} \| \, \|  \S_{2,k}^{-1} \| \, \| \S_{2,k} - \S_{1,k} \| \\
	& \leq \| \S_{1,k}^{-1} \| \, \|  \S_{2,k}^{-1} \| \, \| C_{1,k} \mathcal{P}_{1,k}^- C_{1,k}^* + \mathcal{R}_{1,k} - C_{2,k} \mathcal{P}_{2,k}^- C_{2,k}^* - \mathcal{R}_{2,k} \|.
\end{align*}
That is bounded since $\S_{1,k}$ and $\S_{2,k}$ are invertible. In this context, and analogously to what was previously shown, the following estimate holds:
\begin{equation*}
	\begin{aligned}
		\| & C_{1,k} \mathcal{P}_{1,k}^- C_{1,k}^* +  \mathcal{R}_{1,k} - C_{2,k} \mathcal{P}_{2,k}^- C_{2,k}^* - \mathcal{R}_{2,k} \| \\ 
		& \leq (\| C_{1,k} \|  \| \mathcal{P}_{1,k}^- \| + \| \mathcal{P}_{2,k}^- \| \| C_{2,k} \|) \| C_{1,k} - C_{2,k} \| 
        + \| C_{1,k} \|  \| C_{2,k} \| \| \mathcal{P}_{1,k}^-  - \mathcal{P}_{2,k}^- \| + \| \mathcal{R}_{1,k} - \mathcal{R}_{2,k} \| \\
        & \leq \Big( \| C_{1,k} \|  (\| A_{1,k} \|^2 \| \mathcal{P}_{1,k-1}^+ \| + \| \mathcal{Q}_{1,k} \|) 
        + \| C_{2,k} \|  (\| A_{2,k} \|^2 \| \mathcal{P}_{2,k-1}^+ \| + \| \mathcal{Q}_{2,k} \|) \Big) \| C_{1,k} - C_{2,k} \| \\
        & \quad + \| C_{1,k} \|  \| C_{2,k} \| \Big( ( \| A_{1,k} \| \| \mathcal{P}_{1,k-1}^+ \| + \| A_{2,k} \| \| \mathcal{P}_{2,k-1}^+ \|) \| A_{1,k} - A_{2,k} \| \Big) \\
        & \quad + \| C_{1,k} \|  \| C_{2,k} \| \Big( \| \mathcal{Q}_{1,k} - \mathcal{Q}_{2,k} \| + \| A_{1,k} \| \| A_{2,k} \| \| \mathcal{P}_{1,k-1}^+ - \mathcal{P}_{2,k-1}^+\| \Big) 
        +  \| \mathcal{R}_{1,k} - \mathcal{R}_{2,k} \| .
	\end{aligned}
\end{equation*}

And for the a posteriori covariance error operator, we have
\begin{equation*}
    \begin{aligned}
        \| \mathcal{P}_{1,k}^+ - \mathcal{P}_{2,k}^+ \|  
        & \leq \| (I - \K_{1,k} C_{1,k}) \mathcal{P}_{1,k}^- - (I - \K_{2,k} C_{2,k}) \mathcal{P}_{2,k}^- \| \\
        & \leq \| \mathcal{P}_{1,k}^- - \mathcal{P}_{2,k}^- \| + \| \K_{1,k} \| \| C_{1,k} \| \| \mathcal{P}_{1,k}^- - \mathcal{P}_{2,k}^- \| 
        + \| \K_{1,k} C_{1,k} - \K_{2,k} C_{2,k} \| \| \mathcal{P}_{2,k}^- \| \\
        & \leq (1 + \| \K_{1,k} \| \| C_{1,k} \|) \Big( ( \| A_{1,k} \| \| \mathcal{P}_{1,k-1}^+ \| + \| A_{2,k} \| \| \mathcal{P}_{2,k-1}^+ \|) \| A_{1,k} - A_{2,k} \| \Big) \\
        & \quad + (1 + \| \K_{1,k} \| \| C_{1,k} \|)\Big( \| \mathcal{Q}_{1,k} - \mathcal{Q}_{2,k} \| + \| A_{1,k} \| \| A_{2,k} \| \| \mathcal{P}_{1,k-1}^+ - \mathcal{P}_{2,k-1}^+\| \Big) \\
        & \quad + (\| A_{2,k} \|^2 \| \mathcal{P}_{2,k-1}^+ \| + \| \mathcal{Q}_{2,k} \|) 
        \Big( \| \K_{1,k} \|  \| C_{1,k} - C_{2,k} \| + \| C_{2,k} \|  \| \K_{1,k} - \K_{2,k} \| \Big) \\
        & + \| C_{2,k} \| .
    \end{aligned}
\end{equation*}

With this, grouping terms it can be seen that the explicit constants are
\begingroup
\allowdisplaybreaks
\begin{align*}
        \hat{c}_1 & = (\| y_{2,k} \| + \| \hat\mu_{2,k-1} \| \| C_{2,k} \|)  (\| C_{1,k} \| \| \S_{1,k}^{-1} \| + \| C_{1,k} \| \| C_{2,k} \| )  \\ 
        \hat{c}_2 & = \| A_{2,k} \|^2 \| \mathcal{P}_{2,k-1}^+ \| + \| \mathcal{Q}_{2,k} \| \\
        c_{1,k}^1 & = \| \hat{\mu}_{2,k-1} \| + \hat{c}_2 \| C_{1,k} \|  \\
        & \quad + (\| y_{2,k} \| + \| \hat\mu_{2,k-1} \| \| C_{2,k} \|)  (\| C_{1,k} \| \| \S_{1,k}^{-1} \| + \| C_{1,k} \| \| C_{2,k} \| ) ( \| A_{1,k} \| \| \mathcal{P}_{1,k-1}^+ \| + \| A_{2,k} \| \| \mathcal{P}_{2,k-1}^+ \|) \\
        \Tilde{c}_{2,k}^1 & = \| C_{1,k} \|  (\| A_{1,k} \|^2 \| \mathcal{P}_{1,k-1}^+ \| + \| \mathcal{Q}_{1,k} \|) + \hat{c}_2\| C_{2,k} \|   \\
        c_{2,k}^1 & = \hat{c}_2 \| \S_{2,k}^{-1} \| \| \hat{\mu}_{2,k-1} \| (\| A_{1,k-1} \| + \| A_{2,k-1} \|)  \\
        & \quad + \hat{c}_1 \cdot \hat{c}_2 ( \| A_{1,k} \| \| \mathcal{P}_{1,k-1}^+ \| + \| A_{2,k} \| \| \mathcal{P}_{2,k-1}^+ \|)  ( \| S_{2,k}^{-1}\| + \|  C_{1,k} \| \| \S_{1,k}^{-1} \| \|  \S_{2,k}^{-1} \| \Tilde{c}_{2,k}^1)  \\
        c_{3,k}^1 & = \hat{c}_1 \| \S_{1,k}^{-1} \| \| C_{1,k} \| ( \| A_{1,k} \| \| \mathcal{P}_{1,k-1}^+ \| + \| A_{2,k} \| \| \mathcal{P}_{2,k-1}^+ \|)  ( 1 + \hat{c}_2 \| C_{2,k} \|  \|  \S_{2,k}^{-1} \| \|  C_{1,k} \|) \\
        c_{4,k}^1 & = \hat{c}_1 ( \| A_{1,k} \| \| \mathcal{P}_{1,k-1}^+ \| + \| A_{2,k} \| \| \mathcal{P}_{2,k-1}^+ \|) ( \| C_{1,k} \| \| \S_{1,k}^{-1} \|^2 \| \S_{2,k}^{-1} \|  + \hat{c}_2  \|  C_{1,k} \|) \\
        c_{5,k}^1 & = \hat{c}_2 \| C_{1,k} \| \| \S_{1,k}^{-1} \| \\
        c_{6,k}^1 & = \| \K_{1,k} \| \| C_{1,k} \| \| A_{1,k} \| + \hat{c}_2 \| C_{1,k} \| \| C_{1,k} \| \| \S_{1,k}^{-1} \| \\
        c_{7,k}^1 & = \hat{c}_1 \cdot \hat{c}_2 ( \| A_{1,k} \| \| \mathcal{P}_{1,k-1}^+ \| +  \| A_{2,k} \| \| \mathcal{P}_{2,k-1}^+ \|) \|  C_{1,k} \| \| \S_{1,k}^{-1} \| \|  \S_{2,k}^{-1} \| \| C_{1,k} \|  \| C_{2,k} \|  \| A_{1,k} \| \| A_{2,k} \|.
    \end{align*}
\endgroup
\end{proof}

\subsection{Koopman Kalman Filter (KKF)}

In this subsection, we present the derivation of the filtering algorithm exclusively within the framework of RKHS theory, combined with the Koopman operator, in a meticulous and rigorous manner. From the previous section, the following assumption seems to be reasonable, and a case when they are satisfied is in \cite{olguin2025improvederrorboundskoopman}.

\begin{assu}

    \begin{itemize}
        \item[a)] The observed dynamical system admits a pair $(\X, \Y)$ of invariant spaces such that both are compact with Lipschitz boundaries and, in particular, satisfy the interior cone condition. We denote by $\B_\X$ the $\sigma$-algebra $\B_{\R^n}$ restricted to $\X$, and by $\B_\Y$ the $\sigma$-algebra $\B_{\R^p}$ restricted to $\Y$. Moreover, we consider a probability measure $\mu_\X : \B_\X \to [0,1]$, representing a random variable $X$, and a probability measure $\mu_\Y : \B_\Y \to [0,1]$.
        \item[b)] The functions $\rho_\mathbf{f}$ and $\rho_\mathbf{g}$ admit densities with respect to $\mu_\X$ and $\mu_\Y$, respectively, denoted by \(p_\mathbf{f} : \X \times \X \to \R^+, \, p_\mathbf{g} : \X \times \Y \to \R^+,\) such that for all $\mathbf{x} \in \X$, $\mathcal{A}_1 \in \B_\X$, and $\mathcal{A}_2 \in \B_\Y$,
    \[
        \rho_\mathbf{f}(\mathbf{x}, \mathcal{A}_1) = \int_{\mathcal{A}_1} p_\mathbf{f}(\mathbf{x}, x') \, d\mu_\X(x'),
        \qquad
        \rho_\mathbf{g}(\mathbf{x}, \mathcal{A}_2) = \int_{\mathcal{A}_2} p_\mathbf{g}(\mathbf{x}, y) \, d\mu_\Y(y).
    \]
    Furthermore, $p_\mathbf{f}$ and $p_\mathbf{g}$ are $C^m$, with $m = \lceil n/2 + \nu \rceil$ for some $\nu > 0$. 
    \item[c)] Let $k_\X : \X \times \X \to \R$ be a kernel whose associated RKHS, $\H_\X$, is norm-equivalent to the Sobolev space $H^{n/2 + \nu}$, where $\nu$ is the parameter introduced in Assumption~2. We denote by $\Phi_\X$ its canonical feature map, namely \(\Phi_\X(x) = k_\X(x, \cdot).\)
    \end{itemize}
\end{assu}

The transition measures represent random variables conditional on $X$, which we denote by $X^+$ and $Y$. Specifically, we have
\[
\rho_\mathbf{f}(\mathbf{x}, \mathcal{A}_1) 
= \P \bigl( \mathbf{f}(\mathbf{x}, \cdot) \in \mathcal{A}_1 \bigr) 
= \P \bigl( X^+ \in \mathcal{A}_1 \mid X = \mathbf{x} \bigr),
\]
\[
\rho_\mathbf{g}(\mathbf{x}, \mathcal{A}_2) 
= \P \bigl( \mathbf{g}(\mathbf{x}, \cdot) \in \mathcal{A}_2 \bigr) 
= \P \bigl( Y \in \mathcal{A}_2 \mid X = \mathbf{x} \bigr).
\]
Moreover, due to the invariance of the pair of spaces $(\X,\Y)$, the conditional random variables satisfy
\[
\P \bigl( X^+ \in \X \mid X = \mathbf{x} \bigr) = 1,
\qquad
\P \bigl( Y \in \Y \mid X = \mathbf{x} \bigr) = 1,
\quad \forall \mathbf{x} \in \X,
\]
that is, $X^+ \mid X=\mathbf{x}$ remains in $\X$ almost surely, and $Y \mid X=\mathbf{x}$ remains in $\Y$ almost surely.

At this stage, we set the RKHS over $\Y$ to be \(\H_\Y = (\R^p)^*,\) the topological dual of $\R^p$. Concretely, for every $h \in \H_\Y$ there exists $q \in \R^p$ such that, for all $y \in \R^p$, \(h(y) = \inner{q}{y}.\) Thus, the canonical feature map $\Phi_\Y$ in this space can be chosen up to a bijection if necessary as the identity map.

We now present the kernel Extended Dynamic Mode Decomposition (kEDMD) procedure for approximating the Koopman operator. 
Let \( N \) points be given by \( \{ x_i \}_{i=1}^N \sim \mu_\X^N \), together with two additional sets of points 
\( \{ x_i^+ \}_{i=1}^N \) and \( \{ y_i \}_{i=1}^N \), generated as
\begin{equation*}
    x_i^+ \sim \rho_\mathbf{f} (x_i, \cdot), 
    \quad 
    y_i \sim \rho_\mathbf{g} (x_i, \cdot), 
    \quad i = 1, \dots, N.
\end{equation*}

We define the subspace \(\H_{\X, N} = \mathrm{span} \{ \Phi_\X (x_i) : i = 1, \dots, N \},\) whose canonical basis is given by 
\( \{ \Phi_\X (x_i) : i = 1, \dots, N \} \). The following matrices are introduced:
\begin{equation*}
    \mathbf{X} = (x_{1} \mid \dots \mid x_N), 
    \quad 
    \mathbf{X}^+ = (x_{1}^+ \mid \dots \mid x_N^+), 
    \quad 
    \mathbf{Y} = (y_{1} \mid \dots \mid y_N),
\end{equation*}
\begin{equation*}
    \Phi_{\X, N} (\mathbf{X}) = (k_\X(x_i, x_j))_{i,j = 1}^N, 
    \quad 
    \Phi_{\X, N} (\mathbf{X}^+) = (k_\X(x_i, x_j^+))_{i,j = 1}^N, 
    \quad 
    \Phi_{\Y, N} (\mathbf{Y}) = \bigl( \langle y_i, y_j \rangle \bigr)_{i,j=1}^N.
\end{equation*}

Based on these definitions, we introduce the empirical operators \(C_{X}^N : \H_{\X, N} \to \H_{\X, N}\), \( C_{XX^+}^N : \H_{\X, N} \to \H_{\X, N}\), 
    \(C_{XY}^N : \H_{\X, N} \to (\R^p)^*,\)
defined by
\[
C_X^N = \frac{1}{N} \sum_{j=1}^N \Phi_\X (x_j) \otimes \Phi_\X (x_j), 
\, 
C_{XX^+}^N = \frac{1}{N} \sum_{j=1}^N \Phi_\X (x_j) \otimes \Phi_\X (x_j^+), 
\]
\[
C_{XY}^N = \frac{1}{N} \sum_{j=1}^N \Phi_\X (x_j) \otimes \Phi_\Y (y_j).
\]

The corresponding actions on the canonical basis are given by
\[
C_{X}^N \Phi_\X (x_i) = \frac{1}{N} \sum_{j = 1}^N k_\X(x_i, x_j) \, \Phi_\X (x_j),
\quad 
C_{XX^+}^N \Phi_\X (x_i) = \frac{1}{N} \sum_{j = 1}^N k_\X(x_i, x_j^+) \, \Phi_\X (x_j),
\]
\[
C_{XY}^N \Phi_\X (x_i) = \frac{1}{N} \sum_{j = 1}^N k_\X(x_i, x_j) \, \Phi_\Y (y_j).
\]

These operators can be represented by the matrices 
\( \Phi_{\X,N} (\mathbf{X}) \), \( \Phi_{\X,N} (\mathbf{X}^+) \), and \( \mathbf{Y} \), respectively. 
We now define the empirical approximations of the Koopman operators, 
\(\U_N : \H_{\X,N} \to \H_{\X,N}\) and \(\G_N : \H_{\X,N} \to \H_{\X,N}\), as
\begin{equation*}
    \U_N = (C_X^N)^{-1} C_{XX^+}^N, 
    \qquad 
    \G_N = (C_X^N)^{-1} C_{XY}^N,
\end{equation*}
which are represented in matrix form by
\begin{equation*}
    \mathbf{U}_N = \bigl( \Phi_{\X,N} (\mathbf{X}) \bigr)^{-1} 
    \Phi_{\X,N} (\mathbf{X}^+)^\top,
    \qquad 
    \mathbf{G}_N = \bigl( \Phi_{\X,N} (\mathbf{X}) \bigr)^{-1} 
    \mathbf{Y}^\top.
\end{equation*}

We next derive the dynamics of the embedded system and its associated observations.

\begin{prop} [Embedded system dynamics] \label{prop:embed_sys}
    Let \( \{\mathbf{x}_k\}_{k=0}^T \) denote a trajectory in \(\X\). 
    Its embedding in \(\H_\X\) satisfies
    \begin{align*}
        \Phi_\X (\mathbf{x}_{k+1}) &= C_{X^+|X} \, \Phi_\X (\mathbf{x}_k) + \zeta_k, \\
        \mathbf{y}_k &= C_{Y|X} \, \Phi_\X (\mathbf{x}_k) + \nu_k,
    \end{align*}
    where \( \zeta_k \) and \( \nu_k \) are realizations of random variables taking values in 
    \(\H_\X\) and \(\R^p\), respectively. 
    Both have zero mean and finite second moments. 
    Their covariance operator and covariance matrix are denoted by 
    \(\mathcal{Q}_k : \H_\X \to \H_\X\) and \(\mathcal{R}_k \in \R^{p \times p}\), 
    with \(\mathcal{R}_k\) positive semi-definite. 

    Moreover, if the function \(\mathbf{g}\) satisfies the condition
    \begin{equation}
    \label{eq:condi_g}
        \E\!\left[ \bigl( (\mathbf{g}(\mathbf{x}, \cdot) - \E[\mathbf{g}(\mathbf{x}, \cdot)])^\top v \bigr)^2 \right] = 0, 
        \quad \forall \, \mathbf{x} \in \X \;\; \implies \;\; v = 0,
    \end{equation}
    then \(\mathcal{R}_k \in \R^{p \times p}\) is positive definite.
\end{prop}

This result formalizes the stochastic evolution of the embedded dynamics within the Koopman framework, highlighting how randomness in the system is naturally incorporated into the Hilbert space representation.

\begin{proof}[Proof of Proposition \ref{prop:embed_sys}]
    For the dynamics, we compute
    \begin{equation*}
        \begin{aligned}
            \Phi_\X (\mathbf{x}_{k+1}) 
            &= \E[\Phi_\X (X^+) \mid X = \mathbf{x}_k] 
            + \Phi_\X (\mathbf{f}(\mathbf{x}_k, \mathbf{w}_k)) 
            - \E[\Phi_\X (X^+) \mid X = \mathbf{x}_k] \\
            &= C_{X^+|X} \, \Phi_\X (\mathbf{x}_k) + \zeta_k,
        \end{aligned}
    \end{equation*}
    where \(\zeta_k = \Phi_\X (\mathbf{f}(\mathbf{x}_k, \mathbf{w}_k)) - \E[\Phi_\X (X^+) \mid X = \mathbf{x}_k]\) is a random variable in \(\H_\X\) with bounded covariance operator, denoted by \(\mathcal{Q}_k\). 
    Moreover, \(\zeta_k\) has zero mean since \(\mathbf{f}(\mathbf{x}_k, \mathbf{w}_k) \sim X^+ \mid X = \mathbf{x}_k\), which implies
    \(\E[\zeta_k] 
        = \E[\Phi_\X (\mathbf{f}(\mathbf{x}_k, \mathbf{w}_k))] 
        - \E[\E[\Phi_\X (X^+) \mid X = \mathbf{x}_k]] 
        = 0.\)

    Similarly, for the embedding of the observation, we obtain
    \begin{equation*}
        \begin{aligned}
            \Phi_\Y (\mathbf{y}_{k}) 
            &= \E[\Phi_\Y (Y) \mid X = \mathbf{x}_k] 
            + \Phi_\Y (\mathbf{g}(\mathbf{x}_k, \mathbf{v}_k)) 
            - \E[\Phi_\Y (Y) \mid X = \mathbf{x}_k] \\
            &= C_{Y|X} \, \Phi_\X (\mathbf{x}_k) + \nu_k,
        \end{aligned}
    \end{equation*}
    where \(\nu_k\) is a random variable in \(\R^p\) with covariance matrix
    \[
        \mathcal{R}_k := \E\!\left[ 
        \bigl(\mathbf{g}(\mathbf{x}_k, \cdot) - \E[\mathbf{g}(\mathbf{x}_k, \cdot)]\bigr) 
        \bigl(\mathbf{g}(\mathbf{x}_k, \cdot) - \E[\mathbf{g}(\mathbf{x}_k, \cdot)]\bigr)^\top 
        \right].
    \]
    The matrix \(\mathcal{R}_k\) is positive semidefinite since
    \[
        v^\top \mathcal{R}_k v 
        = \E\!\left[ 
        \bigl((\mathbf{g}(\mathbf{x}_k, \cdot) - \E[\mathbf{g}(\mathbf{x}_k, \cdot)])^\top v\bigr)^2 
        \right] \geq 0.
    \]
    If condition \eqref{eq:condi_g} holds, then \(\mathcal{R}_k\) is positive definite. 
    Finally, \(\nu_k\) has zero mean because 
    \(\mathbf{g}(\mathbf{x}_k, \mathbf{v}_k) \sim Y \mid X = \mathbf{x}_k\), which yields \(\E[\nu_k] = \E[\Phi_\Y (\mathbf{g}(\mathbf{x}_k, \mathbf{v}_k))] - \E[\E[\Phi_\Y (Y) \mid X = \mathbf{x}_k]] = 0.\)
\end{proof}

\begin{rem}
    Condition \eqref{eq:condi_g} ensures that the hypothesis of Theorem \ref{theo:kalman_error} regarding the invertibility of $\mathcal{S}_k$ is satisfied, since $\mathcal{S}_k$ is the sum of a symmetric positive semi-definite matrix and a symmetric positive definite matrix, an argument that will also appear in the derivation of the Kalman gain operator.

    An example where this condition holds is when $\mathbf{g}$ is additive and the noise has a positive definite covariance matrix. That is, if 
    \(\mathbf{g}(\mathbf{x}, \mathbf{v}) = \Tilde{\mathbf{g}}(\mathbf{x}) + \mathbf{v}\), then 
    \(\E[\mathbf{g}(\mathbf{x, \cdot)}] = \Tilde{\mathbf{g}}(\mathbf{x})\), so that
    \(\mathcal{R}_k = \E[\mathbf{v}_k \mathbf{v}_k^\top]\), which is positive definite. Therefore, condition \eqref{eq:condi_g} guarantees the well-posedness of the subsequent filtering recursion.
\end{rem}

\begin{assu}
    The function $\mathbf{g}$ accomplish condition \eqref{eq:condi_g}.
\end{assu}

Following a procedure similar to the one used in \cite{Gebhard2019}, we define
\begin{equation*}
	\hat{\mu}_k = \mathbb{E} [\Phi_\X (\mathbf{x}_k) \mid \mathbf{y}_{1:k}], 
	\quad 
	\mathcal{P}_{k} = \text{Cov}(\Phi_\X(\mathbf{x}_k) - \hat{\mu}_k),
\end{equation*}
with initialization given by \(\hat{\mu}_0 = \hat{\mu}_0^- = \mathbb{E} [\Phi_\X (\mathbf{x}_0)], 
    \quad 
    \mathcal{P}_{0} = \text{Cov}(\Phi_\X (\mathbf{x}_0) - \hat{\mu}_0).\) We also define the prediction step
\begin{equation*}
	\hat{\mu}_{k+1}^- = \mathbb{E} [\Phi_\X (\mathbf{x}_{k+1}) \mid \mathbf{y}_{1:k}], 
	\quad 
	\mathcal{P}_{k+1}^- = \text{Cov}(\Phi_\X (\mathbf{x}_{k+1}) - \hat{\mu}_{k+1}^-).
\end{equation*}

\begin{prop}[On the prediction step] \label{prop:pred_step}
    The quantities $\hat{\mu}_{k+1}^-$ and $\mathcal{P}_{k+1}^-$ satisfy
    \[
        \hat{\mu}_{k+1}^- = C_{X^+|X} \hat{\mu}_k, 
        \quad 
        \mathcal{P}_{k+1}^- = C_{X^+|X} \mathcal{P}_k (C_{X^+|X})^* + \mathcal{Q}_{k+1}.
    \]
\end{prop}

\begin{proof}[Proof of Proposition \ref{prop:pred_step}]
    The result follows directly from the \textit{kernel} Bayes rule \cite{Fukumizu2013KernelKernels}, i.e.,
    \begin{equation*}
        \hat{\mu}_{k+1}^- 
        = \mathbb{E} \big[ \Phi_\X (\mathbf{x}_{k+1}) \mid \mathbf{y}_{1:k} \big] 
        = C_{X^+|X} \, \mathbb{E} \big[ \Phi_\X (\mathbf{x}_{k}) \mid \mathbf{y}_{1:k} \big] 
        = C_{X^+|X} \hat{\mu}_k.
    \end{equation*}
    For the covariance, using the independence of $\zeta_{k+1}$ with respect to $\Phi_\X (\mathbf{x}_k)$, we obtain
    \begin{equation*}
        \begin{aligned}
            \mathcal{P}_{k+1}^- 
            &= \text{Cov} \big( \Phi_\X(\mathbf{x}_{k+1}) - \hat{\mu}_{k+1}^- \big) \\[0.5em]
            &= C_{X^+|X} \, \text{Cov} \big( \Phi_\X (\mathbf{x}_{k}) - \hat{\mu}_{k} \big) \, C_{X^+|X}^* 
            + \text{Cov}(\zeta_{k+1}) \\[0.5em]
            &= C_{X^+|X} \, \mathcal{P}_k \, (C_{X^+|X})^* + \mathcal{Q}_{k+1}.
        \end{aligned}
    \end{equation*}
\end{proof}
This expression is the exact analogue of the prediction step in the classical Kalman filter, with $C_{X^+|X}$ acting as the state-transition operator and $\mathcal{Q}_{k+1}$ as the process noise covariance.

\begin{prop}[On the update step, adapted from Gebhardt et al. \cite{Gebhard2019}] 
\label{prop:update_step}
\label{prop:unbias_kalman_operator}
    The estimator $\hat{\mu}_k$ is unbiased for $\mu_k$ for all $k$, and the Kalman gain operator $\mathcal{K}_k: \R^p \to \H_\X$ is given by
    \begin{equation*}
        \mathcal{K}_k = \mathcal{P}^-_{k} C_{Y|X}^* \big( C_{Y|X} \mathcal{P}^-_{k} C_{Y|X}^* + \mathcal{R}_k \big)^{-1}.
    \end{equation*}
    Moreover, the posterior error covariance operator satisfies \( \mathcal{P}_k = (I - \mathcal{K}_k C_{Y|X}) \, \mathcal{P}^-_{k}\).
\end{prop}

\begin{proof}[Proof of Proposition \ref{prop:update_step}]
    Let $\varepsilon_k^- = \hat{\mu}_k^- - \mu_k \in \H_\X$ denote the a priori error and $\varepsilon_k^+ = \hat{\mu}_k - \mu_k \in \H_\X$ the a posteriori error. Then,
    \begin{align*}
        \varepsilon_k^+ 
                &= \hat{\mu}_{k}^- + \K_{k} (\mathbf{y}_{k} - C_{Y|X} \hat{\mu}_{k}^-) - \mu_k \\
                &= \varepsilon_k^- + \K_{k} \mathbf{y}_{k} - \K_k C_{Y|X} \hat{\mu}_{k}^- + \K_k C_{Y|X} \mu_k - \K_k C_{Y|X} \mu_k \\
                &= ( I - \K_k C_{Y|X} ) \varepsilon_k^- + \K_{k} (\mathbf{y}_{k} - C_{Y|X} \mu_k) \\
                &= ( I - \K_k C_{Y|X} ) \varepsilon_k^- + \K_{k} \nu_k.
    \end{align*}

    \textbf{Step 1: Unbiasedness.}  
    We first prove that $\E[\varepsilon_k^-] = 0$ by induction. For $k=0$, this holds by construction, establishing the base case. Assume it holds for some $k \in \N$, and consider $k+1$. Since 
    \[
        \varepsilon_{k+1}^- = C_{X^+|X} \hat{\mu}_k - C_{X^+|X} \mu_k - \zeta_k,
    \]
    it follows that
    \[
        \E[\varepsilon_{k+1}^-] = C_{X^+|X} \E[\hat{\mu}_k - \mu_k] - \E[\zeta_k] = C_{X^+|X} \E[\varepsilon^+_k],
    \]
    as the random variables $\zeta_k$ are centered.  
    Moreover,
    \[
        \E[\varepsilon^+_k] = ( I - \K_k C_{Y|X} ) \E[\varepsilon_k^-] + \K_{k} \E[\nu_k] = 0,
    \]
    since $\nu_k$ is centered and $\E[\varepsilon_k^-]=0$ by the induction hypothesis. Hence, by induction, $\E[\varepsilon_k^-] = \E[\varepsilon^+_k] = 0$ for all $k$, which proves that $\hat{\mu}_k$ is an unbiased estimator of $\mu_k$.

    \textbf{Step 2: Optimal gain operator.}  
    In the filtering problem, the goal is to minimize the expected quadratic loss
    \[
        \E[\langle \hat{\mu}_{k} - \mu_{k},  \hat{\mu}_{k} - \mu_{k} \rangle ] 
        = \E [ \langle \varepsilon_k^+, \varepsilon_k^+ \rangle ].
    \]
    Since the estimator is unbiased, this is equivalent to minimizing the trace of the posterior covariance operator:
    \[
        \min_{\K_k} \E[\langle \varepsilon_k^+, \varepsilon_k^+ \rangle] 
        = \min_{\K_k} \text{Tr}(\mathcal{P}_{k}).
    \]

    Substituting the expression for $\varepsilon_k^+$, we obtain
    \begin{align*}
        \mathcal{P}_k 
        &= \E\Big[ \big((I - \K_k C_{Y|X})\varepsilon^-_k - \K_k \nu_k \big) 
                 \otimes \big((I - \K_k C_{Y|X})\varepsilon^-_k - \K_k \nu_k \big) \Big] \\
        &= (I - \K_k C_{Y|X}) \, \E[\varepsilon^-_k \otimes \varepsilon^-_k] (I - \K_k C_{Y|X})^* 
         - \K_k \E[\nu_k \otimes \varepsilon^-_k](I - \K_k C_{Y|X})^* \\
        &\quad - (I - \K_k C_{Y|X}) \E[\varepsilon^-_k \otimes \nu_k] \K_k^* 
         + \K_k \E[\nu_k \otimes \nu_k] \K_k^*.
    \end{align*}

    By independence of $\nu_k$ and $\varepsilon_k^-$, and since both are centered, the cross-terms vanish. Thus,
    \[
        \mathcal{P}_k = (I - \K_k C_{Y|X}) \mathcal{P}^-_{k} (I - \K_k C_{Y|X})^* + \K_k \mathcal{R}_k \K_k^*,
    \]
    where $\mathcal{R}_k = \E[\nu_k \nu_k^*]$.

    Minimizing the trace with respect to $\K_k$ yields the first-order condition
    \begin{equation}
        -(I - \K_k C_{Y|X})\mathcal{P}^-_{k}C_{Y|X}^* + \K_k \mathcal{R}_k = 0,
        \label{eq:kalman_gain_rel}
    \end{equation}
    or equivalently,
    \[
        \K_k \big( C_{Y|X}\mathcal{P}^-_{k}C_{Y|X}^* + \mathcal{R}_k \big) = \mathcal{P}^-_{k}C_{Y|X}^*.
    \]

    Since $\mathcal{R}_k$ is symmetric positive definite and $C_{Y|X}\mathcal{P}^-_{k}C_{Y|X}^*$ is symmetric positive semidefinite, the operator inside the parentheses is invertible. Hence,
    \[
        \K_k = \mathcal{P}^-_{k} C_{Y|X}^* \big( C_{Y|X} \mathcal{P}^-_{k} C_{Y|X}^* + \mathcal{R}_k \big)^{-1}.
    \]

    Finally, substituting this expression into the covariance update and using relation \eqref{eq:kalman_gain_rel} gives \(\mathcal{P}_k = (I - \K_k C_{Y|X}) \, \mathcal{P}^-_{k}.\)
\end{proof}

Thus, with closed-form expressions for the Kalman gain operator and the error covariance operators, the recursive equations can be written in terms of the Koopman operator. Using that \(C_{X^+|X} = \U^*\) and \(C_{Y|X} = \G^*\), we obtain
\begin{equation*}
	\begin{aligned}
		\hat{\mu}_{k+1}^- & = \U^* \hat{\mu}_{k}, 
		& \mathcal{P}_{k+1}^- & = \U^* \mathcal{P}_k \U + \mathcal{Q}_{k+1}, \\
		\S_{k+1} & = \G^* \mathcal{P}_{k+1}^- \G + \mathcal{R}_{k+1}, 
		& \K_{k+1} & = \mathcal{P}_{k+1}^- \G \S_{k+1}^{-1}, \\
		\mathcal{P}_{k+1} & = (I - \K_{k+1} \G^*) \mathcal{P}_{k+1}^-, 
		& \hat{\mu}_{k+1} &= \U^* \hat{\mu}_k + \K_{k+1} (\mathbf{y}_{k+1} - \G^* \hat{\mu}_{k+1}^-).
	\end{aligned}
\end{equation*}

We now introduce finite-dimensional approximations based on the developments of the previous section:
\begin{equation*}
	\begin{aligned}
		\hat{\mu}_{N, k+1}^- & = \U^*_N \hat{\mu}_{N, k}, 
		&\mathcal{P}_{N, k+1}^- & = \U^*_N \mathcal{P}_{N,k} \U_N + \mathcal{Q}_{N, k+1}, \\
        \S_{N, k+1} & = \G_N^* \mathcal{P}_{N,k+1}^- \G_N + \mathcal{R}_{N,k+1},
		& \K_{N,k+1} & = \mathcal{P}_{N, k+1}^- \G_N (\G^*_N \mathcal{P}_{N, k+1}^- \G_N + \mathcal{R}_{N, k+1})^{-1}, \\
		\mathcal{P}_{N, k+1} & = (I - \K_{N,k+1} \G_N^*) \mathcal{P}_{N,k+1}^-, 
		& \hat{\mu}_{N,k+1} &= \U_N^* \hat{\mu}_{N,k} + \K_{N,k+1} (\mathbf{y}_{k+1} - \G^*_N \hat{\mu}_{N,k+1}^-).
	\end{aligned}
\end{equation*}

Here, $\mathcal{Q}_{N,k+1}$ and $\mathcal{R}_{N,k+1}$ denote unbiased estimators of $\mathcal{Q}_{k+1}$ and $\mathcal{R}_{k+1}$, respectively, i.e.,
\begin{equation}
	\mathcal{Q}_{N,k+1} = \frac{1}{N-1}\sum_{j=1}^N (z_{1,j} - \bar{z}_1) \otimes (z_{1,j} - \bar{z}_1), \quad 
    \mathcal{R}_{N,k+1} = \frac{1}{N-1}\sum_{j=1}^N (z_{2,j} - \bar{z}_2)  (z_{2,j} - \bar{z}_2)^\top,
	\label{eq: emp_covars}
\end{equation}
where $\{ z_{1,j} \}_{j=1}^N \sim \zeta_k^N$, $\{ z_{2,j} \}_{j=1}^N \sim \nu_k^N$, and \(\bar{z}_i = \frac{1}{N} \sum_{j=1}^N z_{i,j}\). Finally, if $X_0$ is the distribution of the initial condition and $\{ x_j \}_{j=1}^N \sim X_0$, the initialization is given by
\begin{equation}
	\hat{\mu}_{N,0} = \frac{1}{N} \sum_{j=1}^N \Phi_\X(x_{j}), \quad 
    \mathcal{P}_{N,0} = \frac{1}{N-1} \sum_{j=1}^N (\Phi_\X(x_{j}) - \hat{\mu}_{N,0}) \otimes (\Phi_\X(x_{j}) - \hat{\mu}_{N,0}).
	\label{eq: mean_element_covar}
\end{equation}

With this, the derivation of the filter is complete, both in infinite dimension and in its finite-dimensional representation. For the latter, its pseudocode is presented in Algorithm \ref{alg:KKF}, that leads to an algorithm with complexity \(O(T\cdot N^3)\) in time, due to that the algorithm is just matrix multiplication in each iteration of the algorithm, and \(O(T\cdot N^2)\) in space, since for each iteration we have to store matrices of size \(N^2\).

\begin{algorithm}[ht]
\caption{Kalman Koopman Filter (KKF)}\label{alg:KKF}
\begin{algorithmic}[1]
\State \textbf{Input:} Discrete dynamics, $\mathbf{x}_0$ prior on the initial condition, $\mathbf{y}_{1:T}$ observations, $\mathbf{k}:\X \times \X \to \R$ positive definite \textit{kernel}, $N$ dimension of the Koopman operator approximation, and $n_{\text{samples}}$ number of samples to approximate the covariance matrices. 
\State \textbf{Output:} $(\hat{\mathbf{x}}_{k|k})_{k=0}^{T}$ trajectory estimator and $(\hat{\mathbf{P}}^{\mathbf{x}}_{k|k})_{k=0}^{T}$ error covariance matrices. \\
\textbf{Initialization:}
\State $\mathbf{U}_N, \, \Phi_\X (\cdot), \, \mathbf{G}_N, \mathbf{B}_N \gets $ kEDMD($\mu_\X$, $\rho_f$, $\rho_g$, $k$, $N$)
\State $\hat{\mathbf{x}}_{0|0}   \gets \E [\mathbf{x}_0]$ \Comment{Initial condition estimate for $\mathbf{x}$}
\State $\hat{\mathbf{z}}_{0|0}   \gets \mathbf{\Phi}_\X(\hat{\mathbf{x}}_{0|0})$ \Comment{Initial condition estimate for $\mathbf{z}$}
\State $\hat{\mathbf{P}}^\mathbf{x}_{0|0} \gets \E [(\mathbf{x}_0 - \hat{\mathbf{x}}_{0})(\mathbf{x}_0 - \hat{\mathbf{x}}_{0})^\top]$ \Comment{Initial error covariance for $\mathbf{x}$}
\State $\hat{\mathbf{P}}^\mathbf{z}_{0|0} \gets$ Cov($\mathbf{x}_0$, $n_{\text{samples}}, \mathbf{\Phi}_\X$) \Comment{Initial error covariance for $\mathbf{z}$}
\For{$k = 0, \dots, N-1$} \\
    \textbf{Prediction:}
    \State $\hat{\mathbf{x}}_{k+1|k} \gets \E[\mathbf{f}(\hat{\mathbf{x}}_{k|k}, \cdot)]$
    \Comment{A priori estimate for $\mathbf{x}$}
    \State $\hat{\mathbf{z}}_{k+1|k} \gets \mathbf{\Phi}_\X(\hat{\mathbf{x}}_{k+1|k})$
    \Comment{A priori estimate for $\mathbf{z}$}
    \State $\mathbf{Q}_k \gets $ Cov($\mathbf{w}_k$, $n_{\text{samples}}, \mathbf{\Phi}_\X(\mathbf{f}(\hat{\mathbf{x}}_{k|k}, \cdot))$) 
    \Comment{Empirical dynamics covariance for $\mathbf{z}$}
    \State $\mathbf{P}^{\mathbf{z}}_{k+1|k} \gets \mathbf{U} \mathbf{P}^{\mathbf{z}}_{k|k} \mathbf{U}^\top + \mathbf{Q}_k$
    \Comment{A priori error covariance} \\
    \textbf{Update:}
    \State $\hat{\mathbf{y}}_{k+1} \gets \E[\mathbf{g}(\hat{\mathbf{x}}_{k+1|k}, \cdot)]$ 
    \Comment{A priori observation estimate}
    \State $\mathbf{e}_{k+1} \gets \mathbf{y}_{k+1} - \hat{\mathbf{y}}_{k+1}$
    \Comment{A priori error (innovation)}
    \State $\mathbf{R}_{k+1} \gets $ Cov($\mathbf{v}_k$, $n_{\text{samples}}, \mathbf{g}(\hat{\mathbf{x}}_{k+1|k}, \cdot)$) 
    \Comment{Empirical observation covariance for $\mathbf{z}$}
    \State $ \mathbf{S}_{k+1} \gets \mathbf{C} \mathbf{P}^{\mathbf{z}}_{k|k} \mathbf{C}^\top + \mathbf{R}_{k+1}$
    \Comment{Residual covariance for $\mathbf{z}$}
    \State $\mathbf{K}_{k+1} \gets \mathbf{P}^{\mathbf{z}}_{k+1|k} \mathbf{C}^\top$Cholesky$(\mathbf{S}_{k+1})^{-1}$
    \Comment{Kalman gain}
    \State $\hat{\mathbf{z}}_{k+1|k+1} \gets \hat{\mathbf{z}}_{k+1|k} + \mathbf{K}_{k+1} \mathbf{e}_{k+1}$
    \Comment{A posteriori estimate for $\mathbf{z}$}
    \State $\hat{\mathbf{x}}_{k+1|k+1} \gets \mathbf{B}\hat{\mathbf{z}}_{k+1|k+1}$
    \Comment{\textit{Lift back} for the state}
    \State $\mathbf{P}^\mathbf{z}_{k+1|k+1} \gets (\mathbf{I} - \mathbf{K}_{k+1} 
    \mathbf{C}) \mathbf{P}^{\mathbf{z}}_{k+1|k}$
    \Comment{A posteriori error covariance for $\mathbf{z}$}
    \State $\mathbf{P}^\mathbf{x}_{k+1|k+1} \gets \mathbf{B}\mathbf{P}^\mathbf{z}_{k+1|k+1} \mathbf{B}^\top$
    \Comment{\textit{Lift back} for the covariance}
\EndFor
\end{algorithmic}
\end{algorithm}

\subsection{Bound error for the filtering algorithm}

The first step in deriving the error bound is to express the discrepancy in the norm of the elements to be compared, at a given time $k$, in terms of the previous step $k-1$ and the involved operators.

\begin{prop} \label{prop:err_KKF_1}
    For $k \geq 1$, there exist constants $c_{k,j}^i > 0$ with $j \in \{ 1, \dots, 6\}$ and $i \in \{ 1, 2\}$ such that
    \begin{equation*}
        \begin{aligned}
            \| \hat{\mu}_{k} - \hat{\mu}_{N,k}  \| \leq & \; c_{1,k}^1 \| \U - \U_N \| +  c_{2,k}^1 \| \G - \G_N \| + c_{3,k}^1 \| \mathcal{Q}_{k} - \mathcal{Q}_{N,k} \| \\ 
            & + c_{4,k}^1 \| \mathcal{R}_{k} - \mathcal{R}_{N,k} \| + c_{5,k}^1 \| \hat{\mu}_{k-1} - \hat{\mu}_{N,k-1} \| + c_{6,k}^1 \| \mathcal{P}_{k-1} - \mathcal{P}_{N,k-1} \|,
        \end{aligned}
    \end{equation*}
    \begin{equation*}
        \begin{aligned}
            \| \mathcal{P}_{k} - \mathcal{P}_{N,k} \| \leq & \; c_{1,k}^2 \| \U - \U_N \| +  c_{2,k}^2 \| \G - \G_N \| + c_{3,k}^2 \| \mathcal{Q}_{k} - \mathcal{Q}_{N,k} \| \\
            & + c_{4,k}^2 \| \mathcal{R}_{k} - \mathcal{R}_{N,k} \|  + c_{5,k}^2 \| \hat{\mu}_{k-1} - \hat{\mu}_{N,k-1} \| + c_{6,k}^2 \| \mathcal{P}_{k-1} - \mathcal{P}_{N,k-1} \|.
        \end{aligned}
    \end{equation*}
    Moreover, the constants $c_{k,j}^i$ depend on $k$ only through the quantities 
    \(\| \U \|\), \(\| \G \|\), \( \| \mathcal{Q}_{k} \|\), \(\| \mathcal{R}_{k} \|\), \(\| \S_{k}^{-1} \|\), \(\| \hat{\mu}_{k-1} \|\), and \(\| \mathcal{P}_{k-1} \|\).
\end{prop}

\begin{proof}[Proof of Proposition \ref{prop:err_KKF_1}]
    The result follows directly from Theorem~\ref{theo:kalman_error}, noting that $\mathcal{S}_k$ is invertible since $\mathcal{R}_k$ is symmetric positive definite.  
    In this setting we identify
    \begin{equation*}
        A_{1,k} = \U^*, \quad A_{2,k} = \U_N^*, \quad 
        C_{1,k} = \G^*, \quad C_{2,k} = \G_N^*, \quad  
        \mathcal{Q}_{1,k} = \mathcal{Q}_k, \quad \mathcal{Q}_{2,k} = \mathcal{Q}_{N,k},
    \end{equation*}
    \begin{equation*}
        \mathcal{R}_{1,k} = \mathcal{R}_k, \quad \mathcal{R}_{2,k} = \mathcal{R}_{N,k}, \quad 
        y_{1,k} = y_{2,k} = \mathbf{y}_k.
    \end{equation*}
    That is, there is no error due to discrepancies in the observations, since both systems are assumed to share the same observation sequence.
\end{proof}

Now, in a previous work \cite{olguin2025improvederrorboundskoopman}, we prove a bound for the Koopman operator approximation via kEDMD, that is of order \(N^{-1/2}\) in probability and will be key in the main theorems of this work. We recall this result below.

\begin{lem}[Bound for kEDMD]
\label{lemma:kedmd_bound}
Let \( \delta \in (0, 1) \) and \( N \in \mathbb{N} \) such that
\[
\delta > 2 \exp \left( -\frac{N c_1^2}{8\|k\|_\infty} \right),
\]
where \( c_1 \) is the coercivity constant of \( C_X \), that is \(\| C_X \psi\| \geq c_1 \| \psi\|, \, \forall \psi \in \H_\X.\) Then, if \( \U \H_\X \subset \H_\X \), \( \G (\R^p)^* \subseteq \H_\X \) and \(\B (\R^n)^* \subseteq \H_\X\), it holds with probability at least \( (1 - \delta)^4 \) that
\begin{equation*}
    \| \U - \U_N \|, \,\| \G - \G_N \|, \, \| \B - \B_N \| \leq C_\delta N^{-1/2},
\end{equation*}
where
\[
C_\delta = \left( \frac{2}{c_1} + \frac{\sqrt{\|k\|_\infty}}{c_1^2} \right)
\sqrt{8 \|k\|_\infty \ln \left( \frac{2}{\delta} \right)}.
\]
\end{lem}

This error bound for the Koopman operator approximation in this setting with kEDMD, lets prove the error bound for the Kalman filter iteration, since one step of the filter is composed mainly of linear relations, with an inverse of an operator that have to pay attention more rigoursly in the proof. The next proposition shows that the embedded dynamic and error covariance operator can be approximated, as RKHS element and operator, respectively, with error \(O(N^{-1/2})\), in each iteration.

\begin{theo}
    Let $\delta \in (0, 1)$ and $N \in \mathbb{N}$ such that
    \[
    \delta > 2 \text{exp} \left( -\frac{Nc_1^2}{8 \max \{\mathbf{g}_\infty, \|k\|_\infty \} } \right)
    \]
    where
    \[
    \mathbf{g}_\infty \coloneq \sup_{\mathbf{x} \in \X}(\E[\mathbf{g}(\mathbf{x},\cdot)^2] + \E[\mathbf{g}(\mathbf{x},\cdot)]^2)
    \]
    and $c_1$ is the coercivity constant of $C_X$. If $\U \H_\X \subset \H_\X$, $\G (\R^p)^* \subset \H_\X$, and $\B (\R^n)^* \subseteq \H_\X$, then there exist constants $C_{k,\delta}^1$ and $C_{k,\delta}^2$ such that, with probability at least $(1 - \delta)^{4+2(T+1)}$, the following holds:
    \begin{equation*}
        \begin{aligned}
            \| \hat \mu_{k} - \hat \mu_{N,k}  \| \leq  \, C_{k, \delta}^1 N^{-1/2}, \, \| \mathcal{P}_{k} - \mathcal{P}_{N,k}  \| \leq \, C_{k, \delta}^2 N^{-1/2}.
        \end{aligned}
    \end{equation*}
    Where the constants $C_{k,\delta}^j = \Tilde{C}_k^j C_\delta$, with $C_\delta$ being the constant defined in Lemma \ref{lemma:kedmd_bound}, and the $\Tilde{C}_k^j$ are positive and depend on $k$ only through \(\| \U \|\), \(\| \G \|\), \( \| \mathcal{Q}_{k} \|\), \(\| \mathcal{R}_{k} \|\), \(\| \S_{k}^{-1} \|\), \(\| \hat{\mu}_{k-1} \|\), and \(\| \mathcal{P}_{k-1} \|\), with $k \in \{ 0, \dots, T-1\}$.
    \label{theo:theo_KKF_2}
\end{theo}

\begin{proof}[Proof of Theorem \ref{theo:theo_KKF_2}]
Since \(\hat{\mu}_{N,0}\), \(\mathcal{P}_{N,0}\), \(\mathcal{Q}_{N,k}\), and \(\mathcal{R}_{N,k}\) are unbiased estimators of \(\hat{\mu}_{0}\), \(\mathcal{P}_{0}\), \(\mathcal{Q}_{k}\), and \(\mathcal{R}_{k}\), respectively, for $k \in \{1, \dots, T\}$, and 
\[
\begin{aligned}
    \| \mathcal{R}_k \| 
    &= \| \E [(\mathbf{g}(\mathbf{x}_k, \cdot) - \E[\mathbf{g}(\mathbf{x}_k,\cdot)])(\mathbf{g}(\mathbf{x}_k, \cdot) - \E[\mathbf{g}(\mathbf{x}_k,\cdot)])^\top ] \| \\
    &\leq  \E \!\left[ \| \mathbf{g}(\mathbf{x}_k, \cdot) - \E[\mathbf{g}(\mathbf{x}_k,\cdot)] \|^2 \right] \\
    & \leq \E[(\|\mathbf{g}(\mathbf{x}_k,\cdot)\| + \E[\mathbf{g}(\mathbf{x}_k,\cdot)])^2] \\
    & \leq \E[\mathbf{g}(\mathbf{x}_k,\cdot)^2] + 2 \E [\mathbf{g}(\mathbf{x}_k,\cdot)]^2 + \E[\mathbf{g}(\mathbf{x}_k,\cdot)^2] \\
    &= 2(\E[\mathbf{g}(\mathbf{x}_k,\cdot)^2] + \E[\mathbf{g}(\mathbf{x}_k,\cdot)]^2) \\
    & \leq 2 \sup_{\mathbf{x} \in \X}\big(\E[\mathbf{g}(\mathbf{x},\cdot)^2] + \E[\mathbf{g}(\mathbf{x},\cdot)]^2\big) 
    = 2\mathbf{g}_\infty,
\end{aligned}
\]
and similarly \(\|\mathcal{P}_k\| \leq 2 \| k \|_\infty\), as same as in \cite{olguin2025improvederrorboundskoopman}, Hoeffding’s inequality \cite{Pinelis1994OptimumSpaces} yields that with probability at least $(1-\delta)^{2(T+1)}$
\[
    \| \hat{\mu}_{N,0} - \hat{\mu}_{0} \|, \, 
    \|\mathcal{P}_{N,0} - \mathcal{P}_{0} \|, \, 
    \| \mathcal{Q}_{N,k} - \mathcal{Q}_{k} \|, \, 
    \| \mathcal{R}_{N,k} - \mathcal{R}_{k} \| 
    \leq C_\delta N^{-1/2}.
\]

\medskip
\noindent
\textbf{Step 1 (Recursive bounds).}  
Applying Proposition \ref{prop:err_KKF_1}, there exist constants $c_{k,j}^i$ with $j \in \{ 1, \dots, 6\}$, $i \in \{ 1, 2\}$ such that
\begin{equation*}
\begin{aligned}
    \| \hat \mu_{k} - \hat \mu_{N,k} \| 
    &\leq c_{1,k}^1 \| \mathcal{U} - \mathcal{U}_N \| 
    + c_{2,k}^1 \| \mathcal{G} - \mathcal{G}_N \| 
    + (c_{3,k}^1+c_{4,k}^1) C N^{-1/2} \\
    &\quad + c_{5,k}^1 \| \hat \mu_{k-1} - \hat \mu_{N, k-1} \| 
    + c_{6,k}^1 \| \mathcal{P}_{k-1} - \mathcal{P}_{N, k-1} \|,
\end{aligned}
\end{equation*}
and similarly
\begin{equation*}
\begin{aligned}
    \| \mathcal{P}_{k} - \mathcal{P}_{N,k} \| 
    &\leq c_{1,k}^2 \| \mathcal{U} - \mathcal{U}_N \| 
    + c_{2,k}^2 \| \mathcal{G} - \mathcal{G}_N \| 
    + (c_{3,k}^2+c_{4,k}^2) C N^{-1/2} \\
    &\quad + c_{5,k}^2 \| \hat \mu_{k-1} - \hat \mu_{N, k-1} \| 
    + c_{6,k}^2 \| \mathcal{P}_{k-1} - \mathcal{P}_{N, k-1} \|.
\end{aligned}
\end{equation*}

\medskip
\noindent
\textbf{Step 2 (Application of Lemma \ref{lemma:kedmd_bound}).}  
By Lemma \ref{lemma:kedmd_bound}, with probability at least $(1-\delta)^{4+2(T+1)}$, there exist constants $C^1_{1, k, \delta}$ and $C^2_{1, k, \delta}$ such that
\begin{equation*}
\begin{aligned}
    \| \hat \mu_{k} - \hat \mu_{N,k} \| 
    &\leq C_{1,k, \delta}^1 N^{-1/2} 
    + C_{2,k, \delta}^1 \| \hat \mu_{k-1} - \hat \mu_{N, k-1} \| 
    + C_{3,k, \delta}^1 \| \mathcal{P}_{k-1} - \mathcal{P}_{N, k-1} \|,
\end{aligned}
\end{equation*}
and
\begin{equation*}
\begin{aligned}
    \| \mathcal{P}_{k} - \mathcal{P}_{N,k} \| 
    &\leq C_{1,k, \delta}^2 N^{-1/2} 
    + C_{2,k, \delta}^2 \| \hat \mu_{k-1} - \hat \mu_{N, k-1} \| 
    + C_{3,k, \delta}^2 \| \mathcal{P}_{k-1} - \mathcal{P}_{N,k-1} \|.
\end{aligned}
\end{equation*}

\medskip
\noindent
\textbf{Step 3 (Induction).}  
We prove the result by induction on $k$.  

\emph{Base case $k=1$.}  
The claim follows directly from Theorem \ref{theo:kalman_error} applied at $k=1$.  

\emph{Inductive hypothesis.}  
Assume that for some $k \in \mathbb{N}$,
\begin{equation*}
\begin{aligned}
    \| \hat \mu_{k} - \hat \mu_{N,k} \| 
    &\leq C_{1,k,\delta}^1 N^{-1/2} 
    + C_{2,k,\delta}^1 \| \hat \mu_{0} - \hat \mu_{N, 0} \| 
    + C_{3,k,\delta}^1 \| \mathcal{P}_{0} - \mathcal{P}_{N, 0} \|,
\end{aligned}
\end{equation*}
\begin{equation*}
\begin{aligned}
    \| \mathcal{P}_{k} - \mathcal{P}_{N,k} \| 
    &\leq C_{1,k,\delta}^2 N^{-1/2} 
    + C_{2,k,\delta}^2 \| \hat \mu_{0} - \hat \mu_{N, 0} \| 
    + C_{3,k,\delta}^2 \| \mathcal{P}_{0} - \mathcal{P}_{N, 0} \|.
\end{aligned}
\end{equation*}

\emph{Inductive step.}  
For $k+1$, it suffices to show the bound for $\| \hat \mu_{k+1} - \hat \mu_{N,k+1} \|$ (the other case is analogous). From the recursion,
\begin{equation*}
\begin{aligned}
    \| \hat \mu_{k+1} - \hat \mu_{N,k+1} \| 
    &\leq C_{1,k+1,\delta}^1 N^{-1/2} 
    + c_{5,k+1}^1 \| \hat \mu_{k} - \hat \mu_{N, k} \| 
    + c_{6,k+1}^1 \| \mathcal{P}_{k} - \mathcal{P}_{N, k} \|.
\end{aligned}
\end{equation*}
Using the inductive hypothesis,
\begin{equation*}
\begin{aligned}
    \| \hat \mu_{k+1} - \hat \mu_{N,k+1} \| 
    &\leq C_{1,k+1,\delta}^1 N^{-1/2} \\
    &\quad + c_{5,k+1}^1 \big( C_{1,k,\delta}^1 N^{-1/2} + C_{2,k,\delta}^1 \| \hat \mu_{0} - \hat \mu_{N, 0} \| + C_{3,k,\delta}^1 \| \mathcal{P}_{0} - \mathcal{P}_{N, 0} \| \big) \\
    &\quad + c_{6,k+1}^1 \big( C_{1,k,\delta}^2 N^{-1/2} + C_{2,k,\delta}^2 \| \hat \mu_{0} - \hat \mu_{N, 0} \| + C_{3,k,\delta}^2 \| \mathcal{P}_{0} - \mathcal{P}_{N, 0} \| \big).
\end{aligned}
\end{equation*}
This can be reorganized as
\[
    \| \hat \mu_{k+1} - \hat \mu_{N,k+1} \| 
    \leq C_{1,k+1,\delta}^1 N^{-1/2} 
    + C_{2,k+1,\delta}^1 \| \hat \mu_{0} - \hat \mu_{N, 0} \| 
    + C_{3,k+1,\delta}^1 \| \mathcal{P}_{0} - \mathcal{P}_{N, 0} \|.
\]
Hence there exists a constant $C_{k+1,\delta}^1$ such that 
\(\| \hat \mu_{k+1} - \hat \mu_{N,k+1} \| \leq C_{k+1,\delta}^1 N^{-1/2}\), 
which completes the inductive step.

Finally, taking \(C_{k,\delta}^j := \max_{i \in \{1,2,3\}} C_{i,k,\delta}^j, 
\quad k \in \{0, \dots, T\}, \, j \in \{1, 2\}\), the proof is completed.
\end{proof}

We now analyze the error induced by returning to the original dimensional space through the \textit{lifting back} operator $\B : (\R^n)^* \to \X$, first defined in our previous work \cite{olguin2025improvederrorboundskoopman}, as follows. Let \(\Tilde{\phi}\) be the identity function in \(\X\); then \(\Tilde{\phi}\) is the canonical \textit{feature map}, up to an isomorphism, of the kernel \(\Tilde{k}(x,x') = x^\top x'\) in \(\X \subset \R^n\). The operator is defined as \(\B \coloneq (C_{X})^{-1} C_{XX},\) where \(C_{XX} \coloneq \E[ \Phi_\X(X) \otimes \Tilde{\phi}(X)].\) The following lemma was proved in \cite{olguin2025improvederrorboundskoopman}.

\begin{lem}
    \label{lem:lifting_back_op}
    The operator \(\B\) satisfies $\B^* \Phi_\X (\mathbf{x}) = \mathbf{x}$.
\end{lem}

We also define the approximation of \(\B\) as 
\[
\B_N \coloneq (C_X^N)^{-1} C_{XX}^N ,
\]
where
\[
C_{XX}^N \coloneq \frac{1}{N} \sum_{j=1}^N \Phi_\X (X_j) \otimes \Tilde{\phi}(X_j).
\]
In virtue of this, we define the exact estimator of the problem as
\[
\hat{\mathbf{x}}_{k} = \B^* \hat{\mu}_{k},
\]
which is not accessible in practice, and is instead approximated by the estimator
\[
\hat{\mathbf{x}}_{N, k} = \B^*_{N} \hat{\mu}_{N,k}.
\]

This will be the estimator provided by the KKF filtering algorithm. Furthermore, recalling that the a posteriori error covariance matrix of the problem is defined as
\[
\mathbf{P}_{k} = \E[ (\hat{\mathbf{x}}_k - \mathbf{x}_k) (\hat{\mathbf{x}}_k - \mathbf{x}_k)^\top ],
\]
we obtain that
\begin{equation*}
    \begin{aligned}
        \mathbf{P}_{k} & = \E[ (\hat{\mathbf{x}}_k - \mathbf{x}_k) (\hat{\mathbf{x}}_k - \mathbf{x}_k)^\top ] \\
        & = \E[ (\B^* \Phi_\X (\hat{\mathbf{x}}_k) - \B^* \Phi_\X (\mathbf{x}_k)) (\B^* \Phi_\X (\hat{\mathbf{x}}_k) - \B^* \Phi_\X (\mathbf{x}_k))^\top ] \\
        & = \B^* \E[ (\mu_k - \hat{\mu}_{k}) \otimes (\mu_k - \hat{\mu}_{k})] \B \\
        & = \B^* \mathcal{P}_k \B.
    \end{aligned}
\end{equation*}
Accordingly, we define the approximate a posteriori error covariance matrix as
\[
\mathbf{P}_{N, k} = \B^*_{N} \mathcal{P}_{N, k} \B_{N}.
\]
With this, the desired error bound for the filter approximation can be easily derived.

\begin{theo}[KKF Error]
\label{theo:error_KKF_fin}
Let $\delta \in (0, 1)$ and $N \in \N$ such that
\[
    \delta > 2 \text{exp} \left( -\frac{Nc_1^2}{8 \max \{\mathbf{g}_\infty, \|k\|_\infty \} } \right)
    \] 
If $\U \H_\X \subset \H_\X$, $\G (\R^p)^* \subset \H_\X$, and $\B (\R^n)^* \subseteq \H_\X$, then there exist constants $\Tilde{C}^1_{k,\delta}$ and $\Tilde{C}^2_{k,\delta}$ such that, with probability at least $(1 - \delta)^{4 + 2(T+1)}$, it holds that
\begin{equation*}
    \| \hat{\mathbf{x}}_k - \hat{\mathbf{x}}_{N, k} \| \leq \Tilde{C}^1_{k,\delta} N^{-1/2},
    \quad 
    \| \mathbf{P}_k - \mathbf{P}_{N, k} \| \leq \Tilde{C}^2_{k,\delta} N^{-1/2}.
\end{equation*}
\end{theo}

\begin{proof}[Proof of Theorem \ref{theo:error_KKF_fin}]
We begin with the error in the state estimation:
\begin{equation*}
    \| \hat{\mathbf{x}}_k - \hat{\mathbf{x}}_{N, k} \|
    \leq \| \mathcal{B}^* \| \, \| \hat{\mu}_{k} - \hat{\mu}_{N, k}\|
    + \|\mathcal{B}^* - \mathcal{B}^*_{N} \| \, \|\hat{\mu}_{N,k} \|.
\end{equation*}
Invoking Lemma \ref{lemma:kedmd_bound} and Theorem \ref{theo:theo_KKF_2}, we obtain that with probability $(1 - \delta)^{4 + 2(T+1)}$ there exist positive constants $C_\delta, C_{k,\delta}^1$ such that
\begin{equation*}
    \| \hat{\mathbf{x}}_k - \hat{\mathbf{x}}_{N, k} \|
    \leq \| \mathcal{B} \| \, C_{k,\delta}^1 N^{-1/2}
    + \| \hat{\mu}_{N, k} \| \, C_\delta N^{-1/2}.
\end{equation*}
Using $\| \mathcal{B}^* \| = \| \mathcal{B} \|$ and $\| \mathcal{B}^* - \mathcal{B}^*_N \| = \| \mathcal{B} - \mathcal{B}_N \|$, we conclude that \(\| \hat{\mathbf{x}}_k - \hat{\mathbf{x}}_{N, k} \|
    \leq \Tilde{C}_{k,\delta}^1 N^{-1/2},\) where $\Tilde{C}_{k,\delta}^1 := \| \mathcal{B} \| C_{k,\delta}^1 + \| \hat{\mu}_{N,k} \| C_\delta$.

\medskip
We now turn to the a posteriori error covariance matrix. Adding and subtracting $\mathcal{B}^* \mathcal{P}_k \mathcal{B}_N$, and applying the triangle inequality, yields
\begin{equation*}
\begin{aligned}
    \| \mathbf{P}_k - \mathbf{P}_{N, k} \|
    &= \| \mathcal{B}^*\mathcal{P}_k \mathcal{B} - \mathcal{B}^*_N \mathcal{P}_{N, k} \mathcal{B}_N \|  \\
    &\leq \| \mathcal{B} - \mathcal{B}_N \| \, \| \mathcal{P}_k \mathcal{B} \|
    + \| \mathcal{B}_N \| \Big( \| \mathcal{B} - \mathcal{B}_N \| \, \| \mathcal{P}_k \|
    + \| \mathcal{P}_k - \mathcal{P}_{N, k} \| \, \| \mathcal{B}_N \| \Big).
\end{aligned}
\end{equation*}
Applying Lemma \ref{lemma:kedmd_bound} and Theorem \ref{theo:theo_KKF_2} again, it follows that with the same probability there exist positive constants $C_\delta, C_{k,\delta}^2$ such that
\begin{equation*}
    \| \mathbf{P}_k - \mathbf{P}_{N, k} \|
    \leq C_{k,\delta}^2 \| \mathcal{P}_k \mathcal{B} \| N^{-1/2}
    + \| \mathcal{B}_N \| \big( C_\delta \| \mathcal{P}_k \|
    + C_{k,\delta}^2 \| \mathcal{B}_N \| \big) N^{-1/2}.
\end{equation*}
Since \(\| \mathcal{B}_N \| \leq \| \mathcal{B}_N - \mathcal{B} \| + \| \mathcal{B} \|
\leq C_\delta N^{-1/2} + \| \mathcal{B} \|,\) we obtain
\begin{equation*}
\begin{aligned}
    \| \mathbf{P}_k - \mathbf{P}_{N, k} \|
    &\leq \Big(
        C_{k,\delta}^2 \| \mathcal{P}_k \mathcal{B} \|
        + (C_\delta + \| \mathcal{B} \|)
        \big( C_\delta \| \mathcal{P}_k \|
        + C_{k,\delta}^2 (C_\delta + \| \mathcal{B} \|) \big)
    \Big) N^{-1/2}.
\end{aligned}
\end{equation*}
Defining \(\Tilde{C}_{k,\delta}^2 :=
C_{k,\delta}^2 \| \mathcal{P}_k \mathcal{B} \|
+ (C_\delta + \| \mathcal{B} \|)
\big( C_\delta \| \mathcal{P}_k \|
+ C_{k,\delta}^2 (C_\delta + \| \mathcal{B} \|) \big),\) the desired estimate follows: \(\| \mathbf{P}_k - \mathbf{P}_{N, k} \| \leq \Tilde{C}_{k,\delta}^2 N^{-1/2}\). This completes the proof.
\end{proof}

With this bound, together with other developments from previous sections, we can now establish the following result regarding the bias and error of the estimator. The filtering problem, although an optimization-inspired problem, has an important statistical factor in practice. For that, we expect the estimator to be asymptotically unbiased in the number of points taken \(N\) for EDMD and to asymptotically attain the lower Mean Squared Error (MSE). 

\begin{theo}[Asymptotic bias and MSE] \label{theo:KKF_bias_var}
    Let $\delta \in (0, 1)$ and $N \in \N$ such that
\[
    \delta > 2 \text{exp} \left( -\frac{Nc_1^2}{8 \max \{\mathbf{g}_\infty, \|k\|_\infty \} } \right)
    \]
where $c_1$ is the injectivity constant of $C_X$. If $\U \H_\X \subseteq \H_\X$, $\G (\R^p)^* \subseteq \H_\X$ and $\B (\R^n)^* \subseteq \H_\X$, then there exist constants $\hat{C}^1_{k,\delta}$ and $\hat{C}^2_{k,\delta}$ such that, with probability at least $(1 - \delta)^4$, the following holds:
    \begin{enumerate}
        \item The bias of the estimator $\hat{\mathbf{x}}_{N, k}$ for the trajectory $\mathbf{x}_k$ is bounded by $\hat{C}^1_{k,\delta} N^{-1/2}$, that is,
    \begin{equation*}
        \left \| \E \left [ \hat{\mathbf{x}}_{N, k} - \mathbf{x}_k \right] \right \| \leq \hat{C}^1_{k,\delta} N^{-1/2}.
    \end{equation*}
        With this, we have that the estimator \(\hat{\mathbf{x}}_{N, k}\) is asymptotically unbiased for \(\mathbf{x}_k\). 
        \item The estimator $\hat{\mathbf{x}}_{N, k}$ is a $\hat{C}^2_{k,\delta} N^{-1/2}$-approximate minimizer of the filtering problem, that is,
    \begin{equation*}
        \E \left [ (\hat{\mathbf{x}}_{N, k} - \mathbf{x}_k)^\top (\hat{\mathbf{x}}_{N, k} - \mathbf{x}_k) \right] \leq  \E \left [ (\hat{\mathbf{x}}_{k} - \mathbf{x}_k)^\top (\hat{\mathbf{x}}_{k} - \mathbf{x}_k) \right] + \hat{C}^2_{k,\delta} N^{-1/2}.
    \end{equation*}
    With this, we have that the estimator asymptotically attains the lower MSE.
    \end{enumerate}
\end{theo}

\begin{rem}
    With this, we obtain that the KKF approximation \(\hat{\mathbf{x}}_{N,k}\) if the asymptotic Minimum Variance Unbiased Estimator (MVUE) for the real state \(\mathbf{x}_k\).
\end{rem}

\begin{proof}[Proof of Theorem \ref{theo:KKF_bias_var}]
    First, for point 1, it is obtained that
    \begin{equation*}
        \begin{aligned}
            \| \mathbb{E} [\hat{\mathbf{x}}_{N, k} - \mathbf{x}_k] \| 
            & = \| \mathbb{E} [ \mathcal{B}^*_N \hat{\mu}_{N,k} - \mathcal{B}^*_N \hat{\mu}_{k} ] + \mathbb{E} [\mathcal{B}^*_N \hat{\mu}_{k} - \mathcal{B}^* \hat{\mu}_k ] + \mathcal{B}^* \mathbb{E} [ \hat{\mu}_k - \mu_k ] \|.
        \end{aligned}
    \end{equation*}
    Thanks to Proposition \ref{prop:unbias_kalman_operator}, we obtain that $\hat{\mu}_k$ is unbiased for $\mu_k$, with which
    \begin{equation*}
        \begin{aligned}
            \| \mathbb{E} [\hat{\mathbf{x}}_{N, k} - \mathbf{x}_k] \|
            & = \| \mathbb{E} [ \mathcal{B}^*_N \hat{\mu}_{N,k} - \mathcal{B}^*_N \hat{\mu}_{k} ] + \mathbb{E} [\mathcal{B}^*_N \hat{\mu}_{k} - \mathcal{B}^* \hat{\mu}_k ] + \mathcal{B}^* \mathbb{E} [ \hat{\mu}_k - \mu_k ] \| \\
            & = \| \mathbb{E} [ \mathcal{B}^*_N \hat{\mu}_{N,k} - \mathcal{B}^*_N \hat{\mu}_{k} ] + \mathbb{E} [\mathcal{B}^*_N \hat{\mu}_{k} - \mathcal{B}^* \hat{\mu}_k ] \| \\
            & \leq \mathbb{E} [\| \mathcal{B}^*_N \| \| \hat{\mu}_{N,k} - \hat{\mu}_{k}\|] + \mathbb{E}[\| \mathcal{B}^*_N - \mathcal{B}^* \| \| \hat{\mu}_k \|]  \leq \hat{C}^1_{k,\delta} N^{-1/2}.
        \end{aligned}
    \end{equation*}
    where Lemma \ref{lemma:kedmd_bound} and Theorem \ref{theo:theo_KKF_2} have been used again to say that such a constant $\hat{C}^1_{k,\delta}$ exists with probability at least $(1 - \delta)^{4 + 2(T+1)}$. To prove point 2, note that
    \begin{equation*}
        \mathbb{E} [ (\hat{\mathbf{x}}_{N, k} - \mathbf{x}_k)^\top (\hat{\mathbf{x}}_{N, k} - \mathbf{x}_k) ] = \mathbb{E} [\hat{\mathbf{x}}_{N, k}^\top \hat{\mathbf{x}}_{N, k}] - 2 \mathbb{E} [\hat{\mathbf{x}}_{N, k}^\top \mathbf{x}_k] + \mathbb{E}[\mathbf{x}_k^\top \mathbf{x}_k]
    \end{equation*}
    \begin{equation*}
        \mathbb{E} [ (\hat{\mathbf{x}}_{k} - \mathbf{x}_k)^\top (\hat{\mathbf{x}}_{k} - \mathbf{x}_k) ] = \mathbb{E} [\hat{\mathbf{x}}_{k}^\top \hat{\mathbf{x}}_{k}] - 2 \mathbb{E} [\hat{\mathbf{x}}_{k}^\top \mathbf{x}_k] + \mathbb{E}[\mathbf{x}_k^\top \mathbf{x}_k].
    \end{equation*}
    Subtracting both, it is obtained that
    \begin{equation*}
        \begin{aligned}
            \mathbb{E} & [ (\hat{\mathbf{x}}_{N, k} - \mathbf{x}_k)^\top (\hat{\mathbf{x}}_{N, k} - \mathbf{x}_k) ] - \mathbb{E} [ (\hat{\mathbf{x}}_{k} - \mathbf{x}_k)^\top (\hat{\mathbf{x}}_{k} - \mathbf{x}_k) ] \\
            & = \mathbb{E} [(\hat{\mathbf{x}}_{N, k} - \hat{\mathbf{x}}_{k})^\top (\hat{\mathbf{x}}_{N, k} + \hat{\mathbf{x}}_{k})] + 2 \mathbb{E} [(\hat{\mathbf{x}}_{k} - \hat{\mathbf{x}}_{N, k})^\top \mathbf{x}_k].
        \end{aligned}
    \end{equation*}
    Note that since $\hat{\mathbf{x}}_k$ is optimal for the filtering problem, it holds that
    \begin{equation*}
        0 \leq \mathbb{E} [ (\hat{\mathbf{x}}_{N, k} - \mathbf{x}_k)^\top (\hat{\mathbf{x}}_{N, k} - \mathbf{x}_k) ] - \mathbb{E} [ (\hat{\mathbf{x}}_{k} - \mathbf{x}_k)^\top (\hat{\mathbf{x}}_{k} - \mathbf{x}_k) ].
    \end{equation*}
    Therefore, and using Cauchy–Schwarz, it is obtained that
    \begin{equation*}
        \begin{aligned}
            0 & \leq \mathbb{E} [(\hat{\mathbf{x}}_{N, k} - \hat{\mathbf{x}}_{k})^\top (\hat{\mathbf{x}}_{N, k} + \hat{\mathbf{x}}_{k})] + 2 \mathbb{E} [(\hat{\mathbf{x}}_{k} - \hat{\mathbf{x}}_{N, k})^\top \mathbf{x}_k] \\
            & \leq \mathbb{E} [\| \hat{\mathbf{x}}_{N, k} - \hat{\mathbf{x}}_k \| \| \hat{\mathbf{x}}_{N, k} + \hat{\mathbf{x}}_k \|] + 2 \mathbb{E} [\| \hat{\mathbf{x}}_{N, k} - \hat{\mathbf{x}}_k \| \| \mathbf{x}_k \| ] \\
            & \leq \mathbb{E} [\| \hat{\mathbf{x}}_{N, k} - \hat{\mathbf{x}}_k \| ( \| \hat{\mathbf{x}}_{N, k} \| + \| \hat{\mathbf{x}}_k \|)] + 2 \mathbb{E} [\| \hat{\mathbf{x}}_{N, k} - \hat{\mathbf{x}}_k \| \| \mathbf{x}_k \| ]  \leq \hat{C}^2_{k,\delta} N^{-1/2},
        \end{aligned}
    \end{equation*}
    where Lemma \ref{lemma:kedmd_bound} and Theorem \ref{theo:theo_KKF_2} have been used again to say that such a constant $\hat{C}^2_{k,\delta}$ exists with probability at least $(1 - \delta)^{4 + 2(T+1)}$. With this,
    \begin{equation*}
        \mathbb{E} [ (\hat{\mathbf{x}}_{N, k} - \mathbf{x}_k)^\top (\hat{\mathbf{x}}_{N, k} - \mathbf{x}_k) ] - \mathbb{E} [ (\hat{\mathbf{x}}_{k} - \mathbf{x}_k)^\top (\hat{\mathbf{x}}_{k} - \mathbf{x}_k) ] \leq \hat{C}^2_{k,\delta} N^{-1/2}, \quad \forall k \in \{ 0, \dots, T \},
    \end{equation*}
    denoting \(\hat{C}^2_\delta := \sup_{k \in \{ 0, \dots,T\} } \hat{C}^2_{k,\delta}\), it is obtained that \(\mathbb{E} [ (\hat{\mathbf{x}}_{N, k} - \mathbf{x}_k)^\top (\hat{\mathbf{x}}_{N, k} - \mathbf{x}_k) ] \leq \mathbb{E} [ (\hat{\mathbf{x}}_{k} - \mathbf{x}_k)^\top (\hat{\mathbf{x}}_{k} - \mathbf{x}_k) ] + \hat{C}^2_{\delta} N^{-1/2}\), so $\hat{\mathbf{x}}_{N, k}$ is a $\hat{C}^2_{\delta} N^{-1/2}$-minimum of the filtering problem.
\end{proof}

\subsection{Parameters estimation algorithm}

Starting from the existence of a filtering algorithm, it's possible to develop a methodology for estimating constant and unknown parameters in a discrete-time dynamical system.

The idea behind this approach is that if the system or the algorithm effectively incorporates the premise that a parameter must exist which fits the observations, then with the stabilization provided by the Kalman gain operator, the constant parameter should converge to its true value as the filtering algorithm's time iterations go to infinity.

Formally, if the system has $n_p$ parameters to be estimated, which are unknown in practice and denoted by $\mathbf{p} \in \R^{n_p}$, then these are added as states with a zero discrete derivative and an additive noise centered at $0$ with a finite second moment. That is, \(\mathbf{x}_{k+1}^{n+i} = \mathbf{x}_{k}^{n+1} + \Tilde{\mathbf{w}}_k^i, \quad i = 1, \dots, n_p.\)
This noise, $\mathbf{w}_k^i$, will assist the convergence of the parameter through the iterations. It is expected that \(\mathbf{x}_{k}^{n+i} \to \mathbf{p}_{i}, \quad k \to \infty\) and this will be observed in numerical experiments.

This approach is not novel and has already been explored in the literature, for example, in \cite{Deng2013AdaptiveObjects, Jiang2007AEstimation, Kandepu2008ApplyingEstimation}. However, this section proposes a reinterpretation of this methodology for parameter estimation. Furthermore, it aims to establish an equivalence with existing methods in other areas of the literature, such as those used in applied research dealing with parameter uncertainty from observations. A prominent example of these techniques is Markov Chain Monte Carlo (MCMC) algorithms \cite{Sammut2010EncyclopediaLearning}.

As is common in many applications, particularly those of interest in this work—the number of available observations is finite and, in many cases, extremely limited. For this reason, assuming the filtering algorithm can converge to the true parameter in the limit as the number of time iterations tends to infinity is not a valid option in this context.

Therefore, it is proposed to perform filtering in an iterative manner. In this approach, a prior is established as an initial condition, allowing the filtering algorithm to complete all available time iterations. Subsequently, the final element produced by the filter is used as a new initial condition, leveraging the error covariance matrix from the last iteration as both the initial condition's and the dynamic model's covariance matrix. This allows the method to self-regulate its variability.

This process can be repeated as many times as necessary and should, in theory, eventually converge to the original parameter. This behavior will be empirically evaluated in the plots and elements presented below. Moreover, drawing an analogy with the MCMC algorithm, it's possible to leverage process parallelization by using all available cores to generate multiple filters running in parallel. This way, different estimates can be obtained for each of the parameters, potentially improving the accuracy and robustness of the estimation process.

Note that, for each thread or subprocess, a time series is generated, where the time unit corresponds to the number of iterations assigned to the iterative algorithm. At each of these time steps, a different estimate for the parameter is produced. This naturally gives rise to a probability distribution that should concentrate around the true parameter. From this distribution, the first iterations of the algorithm will be discarded, as they will be considered \textit{warm-up} iterations, a procedure commonly used in the context of Markov Chain Monte Carlo as well and as implemented by certain frameworks for this method \cite{Patil2010PyMC:Python, Carpenter2017Stan:Language, Burkner2017Brms:Stan, Abril-Pla2023PyMC:Python}. Algorithm~\ref{alg:ParamEstim} outlines the proposed method for estimating constant parameters, which will be applied in the following section.

\begin{algorithm}[ht]
\caption{ParamEstim($\mathbf{x}_0$, $\mathbf{p}_0$, $\mathbf{y}$, $\text{iters}$)}
\label{alg:ParamEstim}
\begin{algorithmic}[1]
\State \textbf{Input:} $\mathbf{x}_0$ prior over the states, $\mathbf{Q}_0$ prior covariance over the states, $\mathbf{p}_0$ prior over the parameter, $\Tilde{\mathbf{P}}_0$ prior covariance over the parameter, $\text{iters}$ number of iterations.
\State \textbf{Output:} $\hat{\mathbf{p}} \in \R^{(\text{iters}+1)\times n_p}$ parameter estimate at each iteration.
\State $\hat{\mathbf{x}}_{0,:} \gets \hat{\mathbf{x}}_0$.
\State $\hat{\mathbf{p}}_{0,:} \gets \hat{\mathbf{p}}_0$.
\State Set initial prior for KKF as $$\Tilde{\mathbf{x}} \gets (\hat{\mathbf{x}}_{0}, \hat{\mathbf{p}}_{0,:}).$$
\State Set initial covariance for KKF as
\begin{equation*}
    \mathbf{P} = \begin{pmatrix}
        \mathbf{Q}_0 & 0 \\
        0            & \Tilde{\mathbf{P}}_0
    \end{pmatrix}
\end{equation*}
\For{$k=1, \dots, \text{iters}$}
    \State $\hat{\mathbf{x}}, \, \hat{\mathbf{P}} \gets \text{KKF}(\Tilde{\mathbf{x}}, \mathbf{P}, \mathbf{y})$ \Comment{$\mathbf{y}$ denotes the observations.}
    \State $\hat{\mathbf{p}}_{k,:} \gets \hat{\mathbf{x}}_{-1,n_p:}$ \Comment{Last iteration of the filter in the last $n_p$ entries.}
    \State $\Tilde{\mathbf{x}} \gets (\hat{\mathbf{x}}_{0}, \hat{\mathbf{p}}_{k,:})$
    \State Update the covariance matrix as
    \begin{equation*}
        \mathbf{P} = \begin{pmatrix}
            \mathbf{Q}_0 & 0 \\
            0            & \hat{\mathbf{P}}_{n_p:, n_p:}
        \end{pmatrix}
    \end{equation*}
    where $\hat{\mathbf{P}}_{n_p:, n_p:}$ is the square submatrix in the last $n_p \times n_p$ block.
\EndFor
\end{algorithmic}
\end{algorithm}

\section{Numerical results}
This section presents the numerical experiments related to the proposed filter, the error bound, and the parameter estimation algorithm. The code for these experiments can be found on \href{https://github.com/diegoolguinw/kkf_bound.git}{GitHub}\footnote{\href{https://github.com/diegoolguinw/kkf_bound.git}{https://github.com/diegoolguinw/kkf\_bound.git}}. The experiments were conducted on an Apple MacBook M1 with 8GB of RAM, and all random seeds were fixed to 42. For the execution of Koopman Kalman Filter, a Python library was develop as the name of \verb|kkf| and can be found in the following \href{https://github.com/diegoolguinw/kkf}{repository}\footnote{\href{https://github.com/diegoolguinw/kkf}{https://github.com/diegoolguinw/kkf}}.

\subsection{Empirical observation of the error bound}

In the linear case, the Kalman filter provides the exact solution to the filtering problem. Therefore, we can compare the trajectories to observe the evolution of the error as the number of sampled points \(N\) increases.

To visualize the bound, we set
\[
\mathbf{f}(\mathbf{x}, \mathbf{w}) = \mathbf{A} \mathbf{x} + \mathbf{w}, \quad \mathbf{g}(\mathbf{x}, \mathbf{v}) = \mathbf{C} \mathbf{x} + \mathbf{v},
\]
where \(\mathbf{A}\) has random entries in each experiment, with \(\mathbf{A}_{ij} \sim \mathcal{N}(0, 0.5)\). This choice was made to obtain different kinds of dynamics while also avoiding rapid exponential growth. The noises are defined as \(\mathbf{w} \sim \mathcal{N}(\mathbf{0}_{3}, \mathbf{Q})\) and \(\mathbf{v} \sim \mathcal{N}(\mathbf{0}_{2}, \mathbf{R})\) with
\[
\mathbf{Q} = 0.01 \cdot \mathbf{I}_{3\times3}, \quad \mathbf{R} = 0.01 \cdot \mathbf{I}_{2\times2}, \quad \mathbf{C} = \begin{pmatrix}
 1 & 0 & 0 \\
 0 & 1 & 0 
\end{pmatrix}.
\]
In this case, the system is observable with probability \(1\) due to the independence of the entries of \(\mathbf{A}\). Thus, the Kalman filter can identify the true trajectories in a finite amount of time. For the Kernel Kalman Filter (KKF), we used the Matérn kernel with \(\nu = 0.5\) and a width of \(\ell = 10^3\). For the distribution of the state space \(\X\), we used a uniform distribution such that \(X_i \sim \text{Unif}(-1000, 1000)\) for \(i \in \{1,2,3\}\).
For the true initial condition, we set \(\mathbf{x}_0 = (1, 1.5, 2)\). For filtering, we used a prior with the distribution \(\mathbf{x}_0^{\text{prior}} \sim \mathcal{N}((0.1, 0.2, 0.3), 0.1 \cdot \mathbf{I}_{3 \times 3} )\).

\begin{figure}[H]
 \centering
 \includegraphics[width=0.8\linewidth]{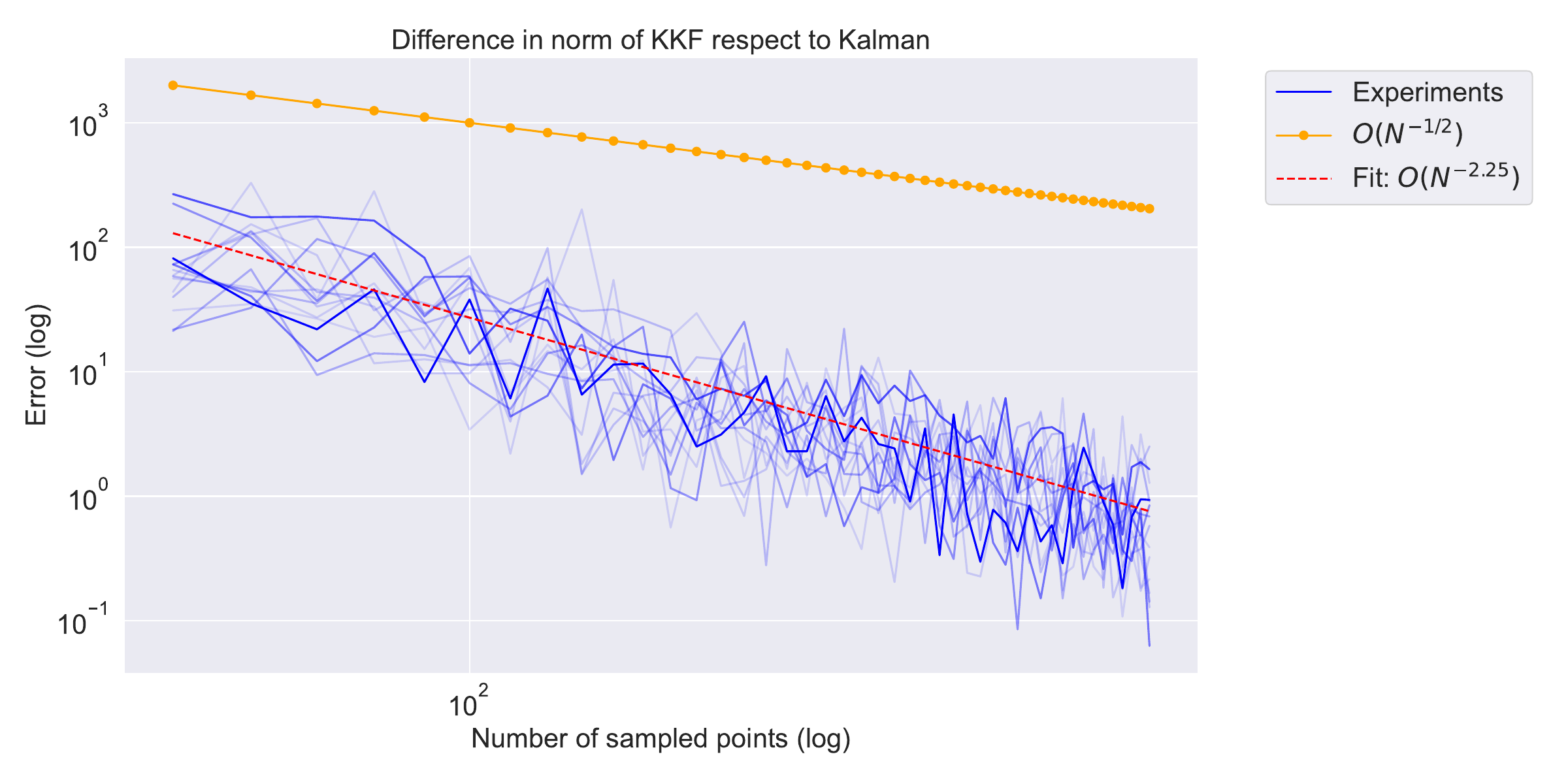}
 \caption{Errors of KKF trajectories with respect to Kalman filter trajectories. The dotted red line represents the best fit curve of the form \(C\cdot N^{\alpha}\) that fits the empirical errors.}
 \label{fig:errors_linear}
\end{figure}

We generated 15 different random matrices \(\mathbf{A}\) to create trajectories over 20 time units. For \(N \in [50, 500]\), we compared the empirical errors with the curve \(C\cdot N^{-1/2}\), where \(C = 8000\). We observed that the error of the KKF trajectories relative to the Kalman filter trajectories, as a function of \(N\), decreases at a rate of \(O(N^{-1/2})\) or even faster, as can be seen in Figure \ref{fig:errors_linear}.

\subsection{Comparison with other nonlinear filters}

This section presents the numerical experiments related to the proposed filter, the error bound, and the parameter estimation algorithm. We test the performance of the Kernel Kalman Filter (KKF) in the context of epidemiological models, which are classical and important examples in filtering and parameter estimation for nonlinear models. We test our approach on one of the most classical epidemiological models: the SIR model. This model simulates the interaction between susceptible and infected individuals who then recover and gain immunity. The model is represented by the following discrete-time equations:

\begin{equation}
 \begin{aligned}
 S_{k+1} & = S_k - \beta S_k I_k^p, \\
 I_{k+1} & = I_k + \beta S_k I_k^p - \gamma I_k, \\
 R_{k+1} & = R_k + \gamma I_k,
  \end{aligned}
  \tag{SIR}
\end{equation}
where \(p\) is the interaction exponent between infected and susceptible individuals and, along with the parameter \(\beta\), represents the model's nonlinearity. When \(p=1\), we recover the classical SIR system, but in many cases, higher exponents are considered to provide a better fit to the behavior of an infectious disease.

Figure \ref{fig:kkf_realization} shows a KKF execution for a SIR system, including the confidence interval for each coordinate of the system. These intervals are induced by the singular values of the matrix \(\mathbf{P}_{N,k}\), which approximates the covariance error matrix. That is, if the Singular Value Decomposition (SVD) of \(\mathbf{P}_{N,k}\) is \(\mathbf{P}_{N,k} = \mathbf{U} \cdot \text{diag}(\Sigma_{k,1}, \dots, \Sigma_{k,n}) \cdot \mathbf{V}^\top,\) where the singular values \(\Sigma_{k,i}\) have been reordered to correspond to the \(i\)-th coordinate of \(\hat{\mathbf{x}}_{N,k}\). By making a normal approximation, the \(1-\alpha\) confidence interval for \(\hat{\mathbf{x}}_{N,k}\) is \(\hat{\mathbf{x}}_{N,k} \pm z_{\alpha} (\Sigma_{k,1}, \dots, \Sigma_{k,n})^{\top}\), where \(z_\alpha\) is the \(\alpha\)-th percentile of a Student's t-distribution.

\begin{figure}[ht]
 \centering
 \includegraphics[width=0.75\linewidth]{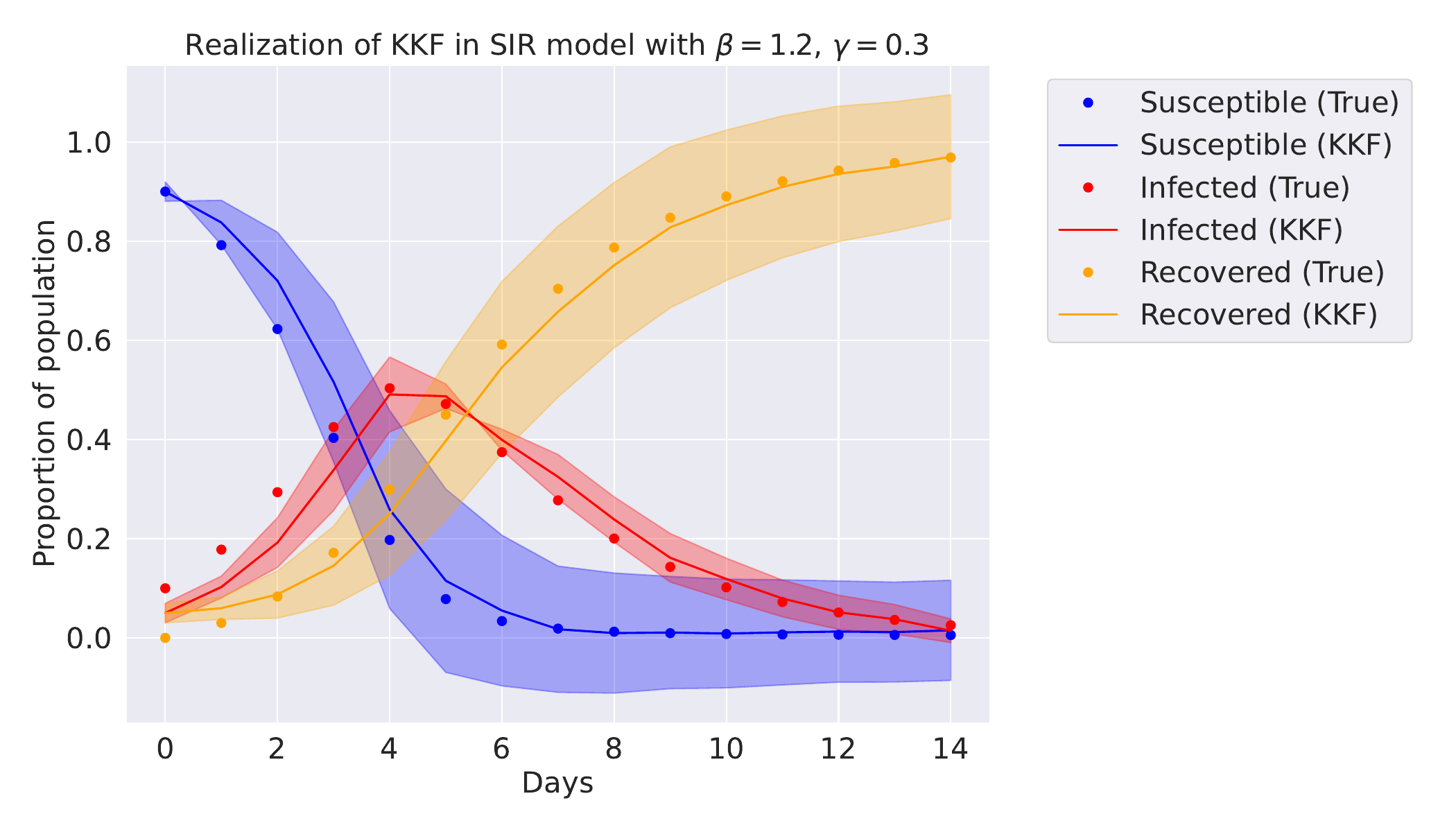}
 \caption{Realization of KKF in a SIR model.}
 \label{fig:kkf_realization}
\end{figure}

We compare the performance of the KKF against other filters from the literature, such as the Extended Kalman Filter (EKF), Unscented Kalman Filter (UKF), and Particle Filters (PF), for the (SIR) system with the following settings:
\begin{itemize}
 \item Setting 1: \(p = 1\), \(\beta \sim \text{Unif}(0.3, 1.5)\) and \(\gamma = 0.3\).
 \item Setting 2: \(p = 2\), \(\beta \sim \text{Unif}(0.3, 5)\) and \(\gamma = 0.3\).
 \item Setting 3: \(p = 3\), \(\beta \sim \text{Unif}(0.3, 10)\) and \(\gamma = 0.3\).
\end{itemize}
In each setting, we performed 10 experiments by sampling a random value for the infection rate \(\beta\) from its respective uniform distribution. We ran the model for 15 time units, adding centered Gaussian noise with a covariance matrix of \(0.01 \cdot \mathbf{I}_{3 \times 3}\). We observed only the second coordinate of the system, corresponding to the infected cases, with additive noise from the distribution \( \mathcal{N}(0, 0.01)\). For the PF and KKF, we show results for different numbers of particles and sampled points, respectively, to better demonstrate the improvement in error and the increase in execution time. The results are shown in Table \ref{tab:SIR_filtering_results}.

\begin{table}[H]
 \centering
 \begin{tabular}{{ccccccc}} 
  \toprule
  \multicolumn{1}{c}{Filtering} &
  \multicolumn{2}{c}{Setting 1} &
  \multicolumn{2}{c}{Setting 2} &
  \multicolumn{2}{c}{Setting 3} \\
  \cmidrule(lr){2-3}\cmidrule(lr){4-5}\cmidrule(lr){6-7}
  algorithm & \(L^2\) error & Ex. time & \(L^2\) error & Ex. time & \(L^2\) error & Ex. time \\
  \midrule
  EKF & 0.256 & 0.002 & 0.289 & 0.002 & 0.303 & 0.002 \\
  UKF & 0.296 & 0.004 & 0.599 & 0.005 & 0.454 & 0.005 \\
  PF (100 particles) & 0.338 & 0.338 & 0.471 & 0.351 & 0.486 & 0.376 \\
  PF (5000 particles) & 0.264 & 16.036 & 0.438 & 15.909 & 0.447 & 16.052 \\
  PF (10000 particles) & 0.264 & 31.827 & 0.436 & 31.748 & 0.432 & 32.064 \\
  KKF (100 points) & \textbf{0.157} & 0.064 & \textbf{0.120} & 0.056 & \textbf{0.094} & 0.077 \\
  KKF (300 points) & \textbf{0.156} & 0.177 & \textbf{0.136} & 0.159 & \textbf{0.082} & 0.184 \\
  KKF (500 points) & \textbf{0.155} & 0.447 & \textbf{0.133} & 0.454 & \textbf{0.081} & 0.596 \\
  \bottomrule
  \end{tabular}
  \caption{Performance of filtering algorithms in the (SIR$_p$) system with different values of the exponent \(p\) and random values for the infection rate \(\beta\). The parameter \(\gamma\) was fixed at \(0.3\) for all cases. The table presents the mean \(L^2\) error and mean execution time, based on 10 experiments for each setting.}
 \label{tab:SIR_filtering_results}
\end{table}

\subsection{Performance and comparison of parameter estimation algorithm}

We now consider more complex epidemiological systems. First, when the dynamics allow recovered individuals to lose immunity, the model changes to the SIRS model, represented by the following equations:

\begin{equation}
 \begin{aligned}
  S_{k+1} & = S_k - \beta S_k I_k + \alpha R_k, \\
  I_{k+1} & = I_k + \beta S_k I_k - \gamma I_k, \\
  R_{k+1} & = R_k + \gamma I_k - \alpha R_k,
 \end{aligned}
 \tag{SIRS}
\end{equation}
When a latent infected population is added, the SEIRS model arises:
\begin{equation}
 \begin{aligned}
  S_{k+1} & = S_k - \beta S_k I_k + \alpha R_k, \\
  E_{k+1} & = E_k + \beta S_k I_k - \delta E_k, \\
  I_{k+1} & = I_k + \delta E_k - \gamma I_k, \\
  R_{k+1} & = R_k + \gamma I_k - \alpha R_k.
 \end{aligned}
 \tag{SEIRS}
\end{equation}

We compare our parameter estimation algorithm, which uses the KKF, with other commonly used Bayesian techniques for parameter estimation, in particular in epidemiological models \cite{molina2025impactofvaccine}. These include MCMC samplers such as Differential Evolution Metropolis (DEMetropolisZ or DEMZ) and the No-U-Turn Sampler (NUTS) \cite{Hoffman2014TheCarlo}. We test the algorithms to estimate the parameters of the SIR, SIRS, and SEIRS models to show that our approach is competitive in terms of error but, more importantly, in terms of execution time. For the MCMC methods, we used the PyMC library \cite{Patil2010PyMC:Python, Abril-Pla2023PyMC:Python}. We chose 1,000 draws for DEMetropolisZ, 100 for NUTS, and 300 iterations for the KKF. Each run was performed in parallel on 8 cores, meaning that 8 parameter estimations were run simultaneously on different cores. In our experiments, we only observed the susceptible and infected populations, which correspond to the first and second coordinates of the system. The initial priors for parameter estimation are shown in Table \ref{tab:param_estim_priors}.

\begin{table}[H]
 \centering
 \begin{tabular}{c|c|c}
 \toprule
 Model & Parameters & Initial prior \\
 \midrule
 SIR (\(p=1\)) & \((\alpha, \beta)\) & \(\mathcal{N}((0.1, 0.1), 0.01 \cdot I_{2\times2})\) \\
 SIRS & \((\alpha, \beta, \gamma)\) & \(\mathcal{N}((0.1, 0.1, 0.1), 0.01 \cdot I_{3\times3})\) \\
 SEIRS & \((\alpha, \beta, \gamma, \delta)\) & \(\mathcal{N}((0.1, 0.1, 0.1, 0.1), 0.01 \cdot I_{4\times4})\) \\
 \bottomrule
 \end{tabular}
 \caption{Initial priors used for parameter estimation across all methods (DEMetropolisZ, NUTS, and KKF).}
 \label{tab:param_estim_priors}
\end{table}

\begin{figure}[ht]
 \centering
 \includegraphics[width=0.5\linewidth]{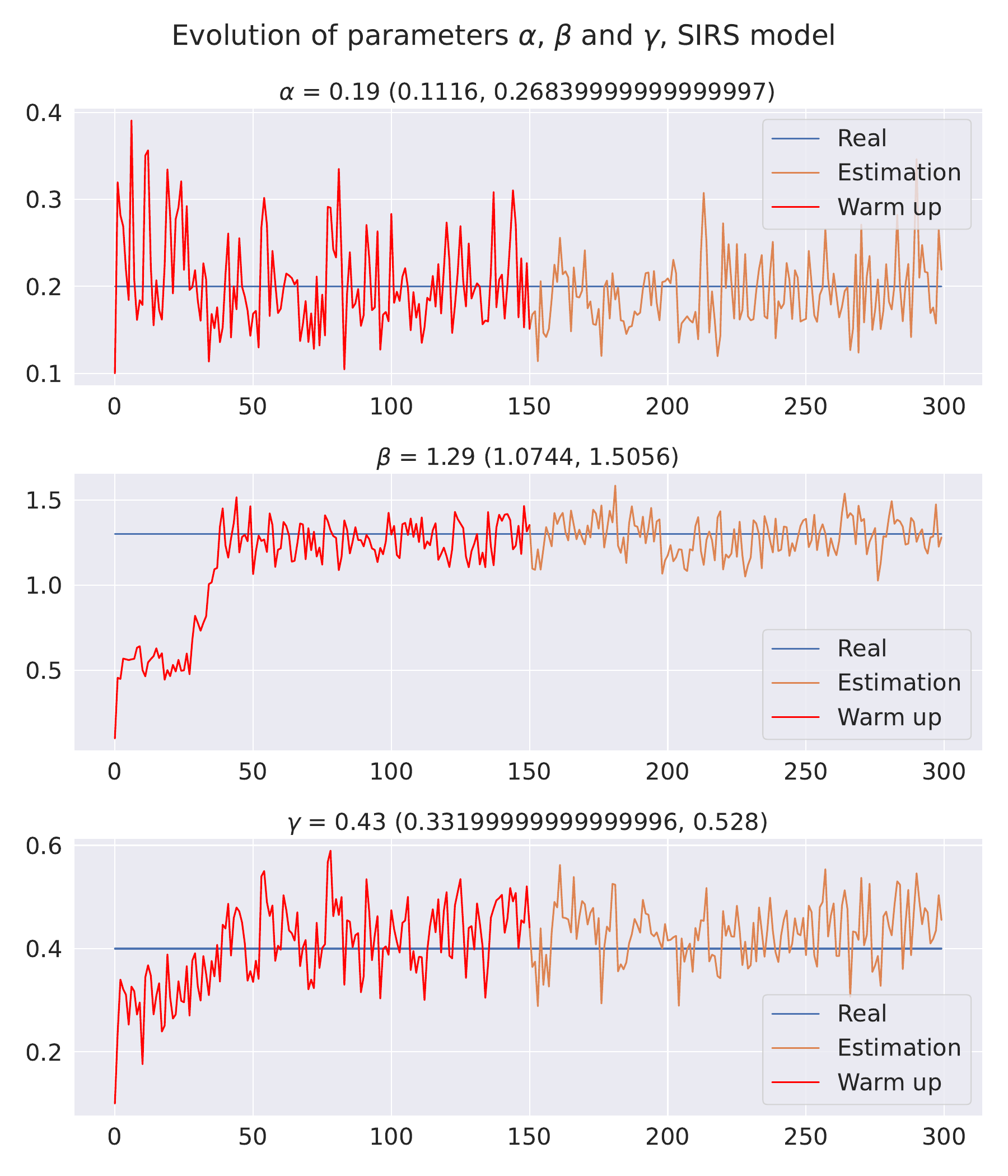}
 \caption{SIRS system parameter evolution in the KKF estimation. The first iterations used for \textit{warm-up} are shown in red. The later iterations for the distribution are shown in orange. The constant true parameter is shown as a blue line. In this case, we can see the successful estimation as the final iterations accumulate around the true parameters.}
 \label{fig:params_evolution_sirs}
\end{figure}

As the parameter estimation iterations proceed, we can see the evolution of the parameters from their initial mean values towards a steady state, which, in the best case, should be a distribution centered on the true parameter. The first iterations are usually not used or taken into account for the final parameter density estimation; these are typically denoted as \textit{warm-up} (or \textit{burn-in}) iterations. Figure \ref{fig:params_evolution_sirs} shows the evolution of the three parameters during the KKF estimation for the SIRS experiment, where the first 150 iterations are used for \textit{warm-up}.

From the final iterations, we can construct a probability density for the parameters based on their evolution. This is done for each process running in parallel. Finally, we take the overall mean to create a more robust density estimation. This can be seen for KKF and NUTS in Figure \ref{fig:params_densities_sirs}. From this, we can generate a parameter estimate with a confidence interval by making a normal approximation of the density. This is a reasonable approximation, given the form of the densities observed in the experiments.

\begin{figure}
 \centering
 \includegraphics[width=0.95\linewidth]{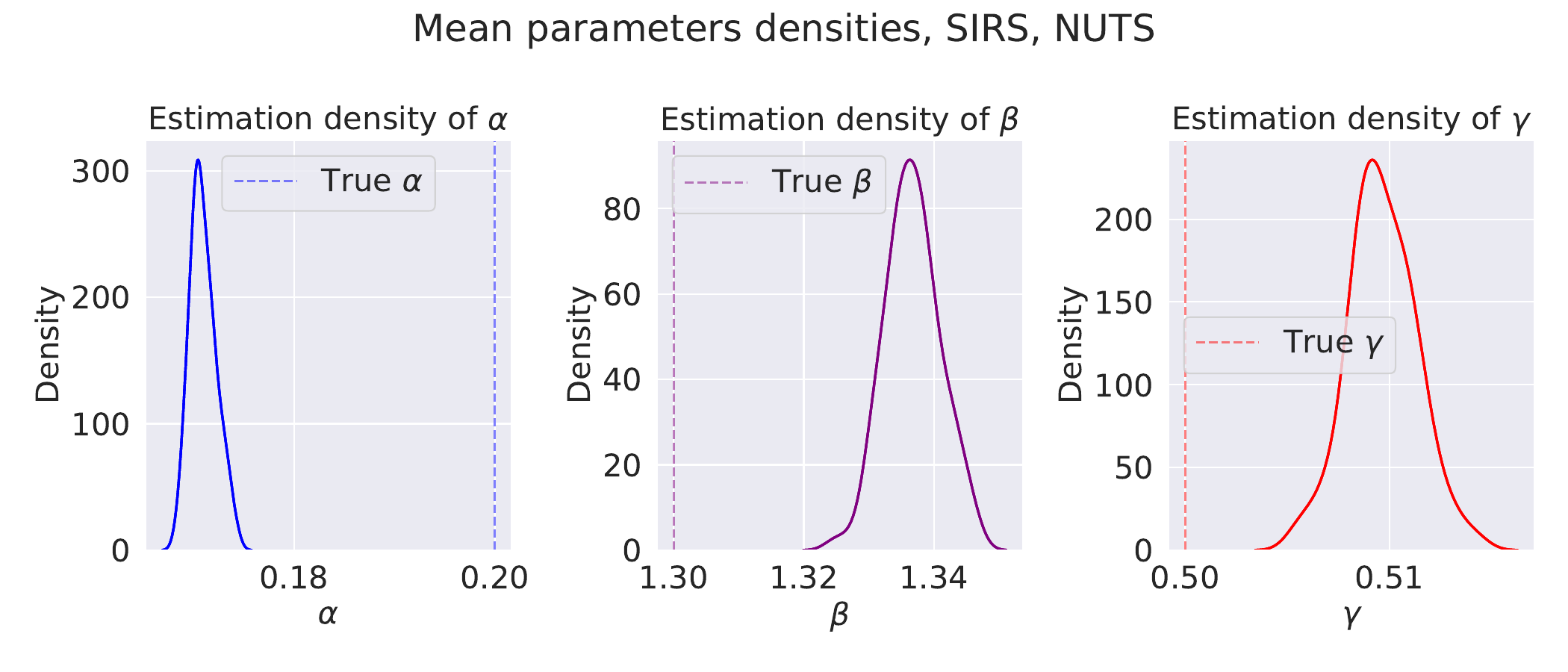}
 \includegraphics[width=0.95\linewidth]{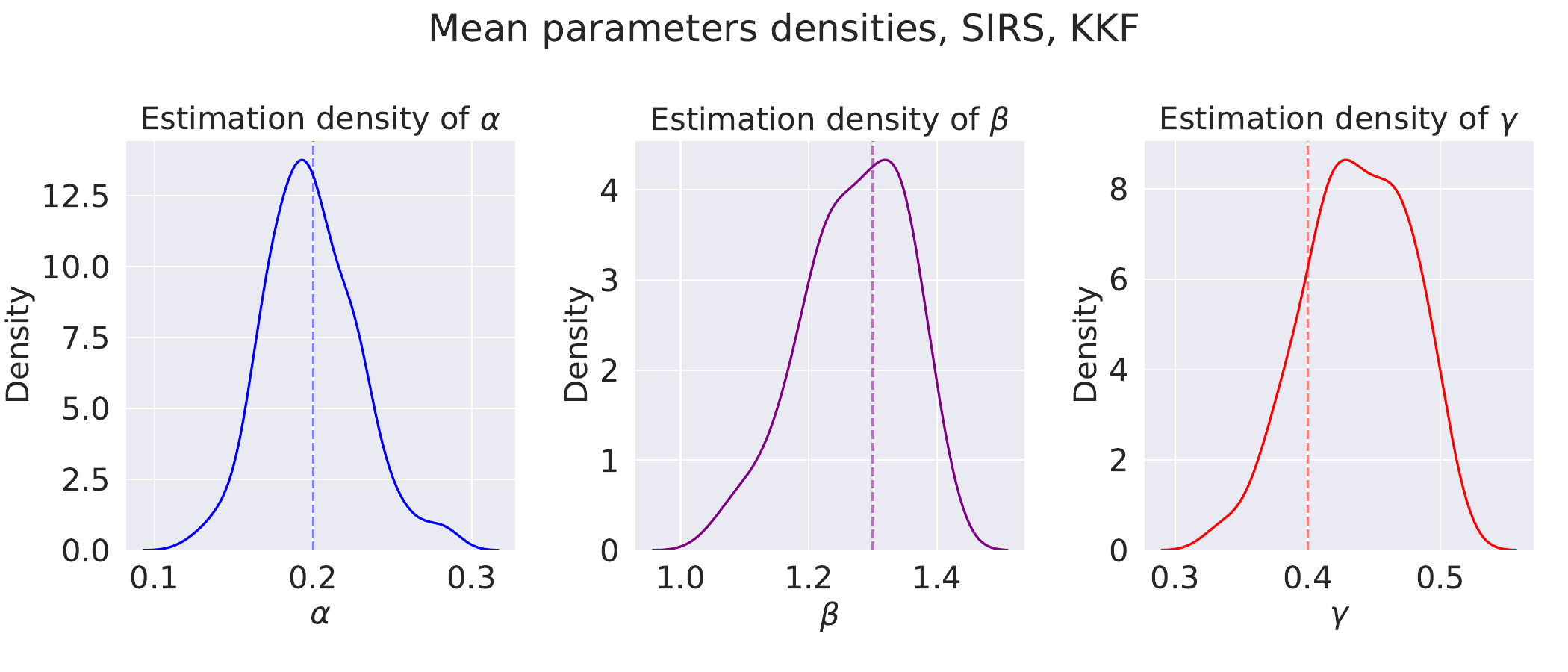}
 \caption{Mean parameter densities for the three parameters of the SIRS system. The mean is taken across all cores that are running a process in parallel. The dashed vertical line indicates the true parameter, which should correspond to, or be near, the mode of the distribution.}
 \label{fig:params_densities_sirs}
\end{figure}

\begin{table}
    \centering
    \begin{tabular}{lccccc}
        \toprule
        \textbf{Method} & \textbf{Model} & $\alpha$ & $\beta$ & $\gamma$ & $\delta$ \\
        \midrule
        \multirow{3}{*}{\textbf{DEMZ}} & SIR & - & $1.58 \, (0.95, 2.22)$ & $0.69 \, (0.23, 1.16)$ & - \\
        & SIRS & $0.39 \, (-0.29, 1.08)$ & $1.49 \, (0.89, 2.08)$ & $0.63 \, (0.20, 1.05)$ & - \\
        & SEIRS & $0.40 \, (-0.15, 0.95)$ & $1.67 \, (0.99, 2.36)$ & $0.48 \, (0.33, 0.63)$ & $0.44 \, (0.15, 0.73)$ \\
        \cmidrule{1-6}
        \multirow{3}{*}{\textbf{NUTS}} & SIR & - & $1.34 \, (1.31, 1.36)$ & $0.45 \, (0.44, 0.46)$ & - \\
        & SIRS & $0.17 \, (0.16, 0.18)$ & $1.34 \, (1.31, 1.36)$ & $0.51 \, (0.50, 0.52)$ & - \\
        & SEIRS & $0.15 \, (0.13, 0.17)$ & $1.27 \, (1.19, 1.35)$ & $0.47 \, (0.43, 0.51)$ & $0.39 \, (0.37, 0.40)$ \\
        \cmidrule{1-6}
        \multirow{3}{*}{\textbf{KKF}} & SIR & - & $1.38 \, (1.30, 1.46)$ & $0.47 \, (0.45, 0.49)$ & - \\
        & SIRS & $0.2 \, (0.14, 0.26)$ & $1.29 \, (1.15, 1.43)$ & $0.44 \, (0.36, 0.52)$ & - \\
        & SEIRS & $0.09 \, (0.07, 0.11)$ & $0.99 \, (0.93, 1.05)$ & $0.37 \, (0.35, 0.39)$ & $0.63 \, (0.57, 0.69)$ \\
        \cmidrule{1-6}
        \textbf{True par.} & All models & $0.2$ & $1.3$ & $0.4$ & $0.5$ \\
        \bottomrule
    \end{tabular}
    \caption{Results of parameter estimation of epidemiological models for different algorithms.}
    \label{tab:param_estim_params}
\end{table}

An error in the estimated parameters does not necessarily imply a significant error in the resulting trajectories. Furthermore, for some nonlinear systems, different parameter configurations can produce similar paths. To test this, we propagate the parameter distributions obtained from the different algorithms through the dynamical system, generating random trajectories that should be close to the true trajectories.

That is, if \(\mu_\theta\) is the empirical parameter probability distribution generated by the algorithms, and \(\theta_i \sim \mu_\theta\) is a sample of the parameter, a sample for a random trajectory is the one generated by \(\mathbf{f}\) with parameters \(\theta_i\). In a sense, this is a methodology to generate samples of the pushforward distribution of \(\mu_\theta\) through \(\mathbf{f}\). We exemplify this in the case of the SIRS model in Figure \ref{fig:SIRS_traj_params}.

\begin{figure}
 \centering
 \includegraphics[width=0.99\linewidth]{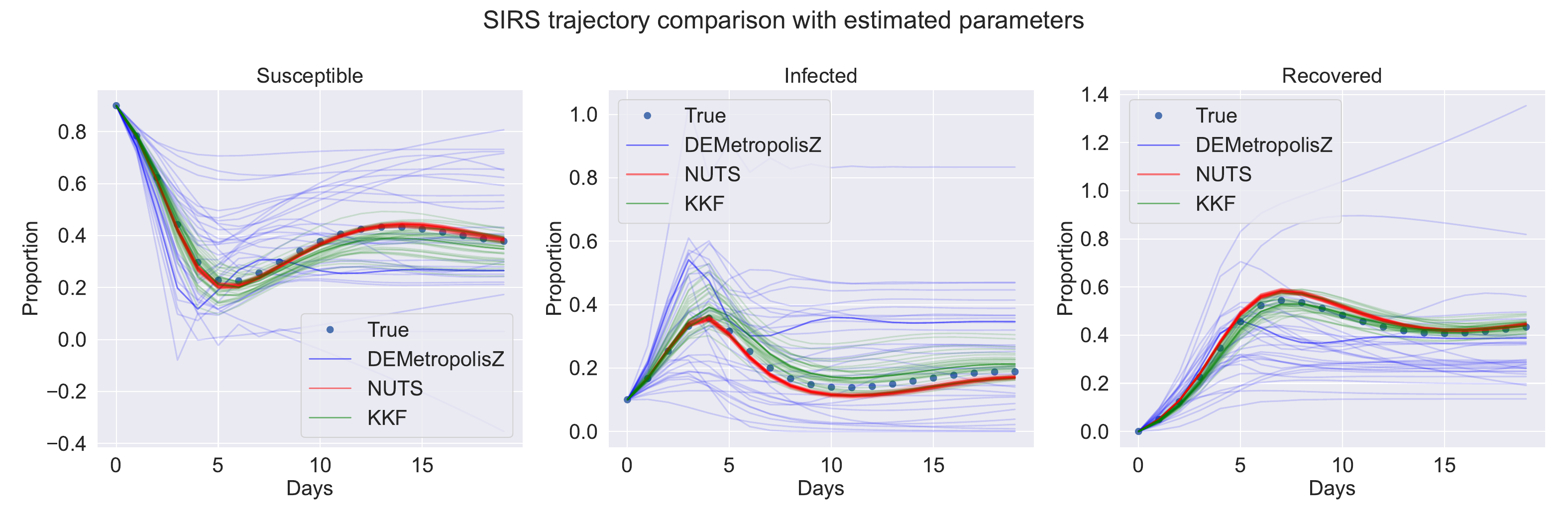}
 \caption{Samples of the trajectories from the parameter distribution generated by each parameter estimation algorithm for the SIRS model.}
 \label{fig:SIRS_traj_params}
\end{figure}

With this, we can measure the mean \(L^2\) error for 30 trajectories sampled from the distribution of the parameters propagated to the dynamical system. The results, along with the mean execution time, can be found in Table \ref{tab:param_estim_error_traj}, where it can be seen that NUTS is a particularly expensive algorithm in terms of time.

\begin{table}
    \centering
    \begin{tabular}{lccc}
        \toprule
        \textbf{Method} & \textbf{Model} & \textbf{\(L^2\) errors} & \textbf{Execution Time in seconds} \\
        \midrule
        \multirow{3}{*}{\textbf{DEMetropolisZ}} & SIR & $0.73$ & $96.61$ \\
        & SIRS & $1.15$ & $74.05$ \\
        & SEIRS & $1.12$ & $\textbf{62.59}$ \\
        \cmidrule{1-4}
        \multirow{3}{*}{\textbf{NUTS}} & SIR & $0.17$ & $341.98$ \\
        & SIRS & $0.16$ & $599.17$ \\
        & SEIRS & $0.19$ & $1654.34$ \\
        \cmidrule{1-4}
        \multirow{3}{*}{\textbf{KKF}} & SIR & $0.19$ & $\textbf{62.43}$ \\
        & SIRS & $0.37$ & $\textbf{61.35}$ \\
        & SEIRS & $0.24$ & $94.87$ \\
        \bottomrule
    \end{tabular}
    \caption{Comparison of methods based on mean \(L^2\) errors and mean execution time in seconds. The mean \(L^2\) errors were calculated from 30 trajectories generated from the probability distribution generated by estimation algorithms.}
    \label{tab:param_estim_error_traj}
\end{table}

\section{Discussion}
At first, this methodology, based on the Koopman operator, allows for the construction of an infinite-dimensional filtering algorithm that is completely based on the Kalman filter but for cases where we are working in Hilbert spaces. This is not only a methodology for constructing a filtering algorithm but also a way to understand the essence of the filtering problem. This is because, for some cases, there exists an equivalent infinite-dimensional optimization problem, which means that the nonlinear filtering problem can be completely reduced to the Kalman filter. This also corroborates the insight of Maurel and Michel \cite{Maurel1984DesFinie} that the nonlinear filtering problem generally does not have finite-dimensional solutions. Furthermore, the finite-dimensional solutions can be characterized by the rank of the Koopman operator, which offers a new perspective in the literature.

Despite this filtering technique being used in other publications, these works are completely empirical and do not provide a theoretical guarantee that the approximate solution converges to the true filtering solution as the number of points taken for EDMD increases. Therefore, we differentiate this work from the existing literature by providing a rigorous proof of convergence for this technique. This proof is also extendable to similar techniques if they have an associated error bound.

When we take a finite approximation of the filtering algorithm with kEDMD, the operator norm error is naturally bounded by \(O(N^{-1/2})\). This is because the operator's approximation is done by averages, and this convergence rate is standard in the literature, particularly with concentration bounds like Hoeffding's inequality. A novelty of this work is that we use the injectivity of the operator \(C_X\) to derive a new constant and, with it, a new bound.

This kind of bound was previously seen in Particle Filters, where the error is also \(O(N^{-1/2})\) where $N$ in this case is the number of points taken for kEDMD. However, in this case, we have two advantages. The first is a more understandable mechanism for the filtering problem, since the Particle Filter simply takes means, and the second is better results in practice and experiments. The second, and most important, advantage is the execution time. Our new algorithm appears to be much more efficient. This is generally due to the Koopman operator's capacity to capture the entire behavior of the dynamics, which can be better than a simple Monte Carlo approximation, as Particle Filters do.

In terms of the experimental results, we observed that the Kernel Kalman Filter (KKF) is superior to the other filtering algorithms with respect to both error and execution time. This is a significant finding because, in filtering tasks, Particle Filters (PF) have been dominant in recent years, with ongoing improvements to their techniques, yet they come with a high computational cost. Our technique, the KKF, has the ability to outperform these types of filters in the tested scenarios, which are also of public importance, particularly in healthcare due to the recent pandemic.

Other filtering algorithms, which are cheaper than Particle Filters and the KKF, such as the Extended Kalman Filter (EKF) and the Unscented Kalman Filter (UKF), do not have the ability to capture the nonlinearity of the systems. This is because they are based on a naive attempt to return to the linear case. While they are computationally cheaper, resulting in lower execution times, they yield poor results in terms of error. Unlike Particle Filters and the KKF, they lack the possibility for improvement by trading higher computational cost for better accuracy.

Another important feature of this methodology, and an advantage with respect to other filters, is that all the steps of the filtering algorithms tested (EKF, UKF, and PF) are performed online. Consequently, we cannot, in principle, save computation, implying that the online procedure is inevitably expensive. However, with KKF, we have the possibility to precompute the Koopman operator approximation, saving the most expensive part of the procedure, which is the generation of the operator's approximation. With this, the only remaining part is the generation of the covariance matrices at each step. Furthermore, a single Koopman operator approximation can serve for many KKF applications on the same system.

Our parameter estimation methodology has several improvements compared to other parameter estimation methodologies that use filters. For instance, the parameter update is performed using the error matrix, which is set as the covariance of the prior for the parameter's initial condition in each filter execution. This proves to be efficient in time and, more importantly, allows for the self-regulation of the error in the estimation, iteration by iteration. The most important insight with respect to this is that our method is very competitive with Markov Chain Monte Carlo (MCMC) methods, particularly the No-U-Turn Sampler (NUTS), which has shown excellence in parameter estimation for dynamical systems. Our algorithm, once again, wins in terms of execution time.

In the tested systems, we achieved similar errors to NUTS but with better execution times, representing an improvement of an order of magnitude. On the other hand, we have similar execution times to DEMetropolisZ but with significantly better error terms, which reflects the algorithm's capability to strongly compete with these other two methods. Furthermore, it has the capability to perform the same task as these two algorithms, which in principle come from other types of techniques, such as MCMC, a method that excels in parameter estimation. This is particularly relevant in epidemiology, where the systems tested are crucial for estimating important factors and making predictions. Since the KKF outperforms other filtering algorithms in filtering tasks, it is reasonable to expect that it can also outperform those algorithms in parameter estimation tasks. These results can be observed not only in the parameter errors but also in the trajectories. DEMetropolisZ produces highly dispersed curves, leading to large errors, while NUTS and the KKF have similar trajectories with comparable errors with respect to the true trajectories.

It is not the first time that filtering algorithms have been used for parameter estimation, but this new technique, which is in principle completely empirical but has yielded good results, proves to be better and on par with some other parameter estimation methodologies.

The reason why the KKF is cheaper than NUTS, for example, is the same as discussed previously for filtering: we can precompute the Koopman operator approximation and then run different realizations of parameter estimation. Thus, precomputing the Koopman operator approximation in a sequential manner before estimating parameters saves a lot of computation time. This is because in the iterations of the algorithm, we only perform matrix multiplications in parallel, which are very cheap, particularly when using a GPU.

\section{Conclusion}
In this article we proposed a new perspective for the filtering problem, using the Koopman operator to reformulate the problem in infinite dimension, showing that a nonlinear problem can be seen as a infinite dimensional one, but linear. With this, we propose an approximation of the filtering problem solution and then apply the classical Kalman Filter. This accomplish an error bound of order \(O(N^{-1/2})\), that is natural since the approximation is based in averages of operators and concentration bounds, as Hoeffding's inequality. But, in the way, we derived an error decomposition for Kalman rules between two different systems, then derived the filter, Kalman Koopman Filter (KKF), with the rigor of working in Reproducing Kernel Hilbert Spaces and using Koopman operator properties, Kernel Bayes rule and other classical techniques in functional analysis.

The new filtering algorithm KKF shows good performance for linear systems and nonlinear systems as SIR, outperforming others in literature as EKF, UKF and Particle Filters, not only in precision, but especially in execution time, showing that the KKF algorithm we propose in this work is really efficient. This efficient comes from the fact that Koopman operator can be precomputed and the only apply matrices multiplications, that are more cheap in practice.

Also, we show that the proposed parameter estimation algorithm is effective and is competitive with respect to MCMC methods, that are the state of the art in parameter estimation for some systems and models, in particular in the class of tested systems, and in others used in other applications than epidemiology. 

With this, this article goes beyond than a new filtering algorithm, proposing a new way to understand completly the essence of the filtering problem through the Koopman operator, and with this generating a new idea for some other Koopman operator based algorithms that can be usefull in engineering, science and other fields, promising to be more efficient and more understanble in practice and theory.

\section*{Acknowledgements}
This work was supported by Center for Mathematical Modeling (CMM) Basal fund FB210005 for center of excellence, FONDECYT 1240200, FONDAP/1523A0002, ECOS240038 and DO210001 all from ANID-Chile, and Joint Cooperation Fund Chile-Mexico 2022: 'Mathematical
Modeling of Epidemic Processes, Incorporating Population Structure, Regional Distribution and
Risk Groups'.

\bibliographystyle{plainnat}
\bibliography{references}
\end{document}